\documentclass[10pt,a4paper,
]{article}
\usepackage[
	a4paper, scale=0.85,
]{geometry}

\usepackage{amsmath}
\usepackage{amsfonts}
\usepackage{amsthm}
\usepackage{amssymb}
\usepackage{latexsym}
\usepackage{fixmath}
\usepackage{mathdots}

\usepackage{xcolor}
\usepackage{graphicx}
\usepackage{subfig}

\usepackage{siunitx}

\usepackage{hyperref}
\hypersetup{colorlinks}  
\usepackage[capitalize]{cleveref}
\crefname{assumption}{assumption}{assumptions}
\Crefname{figure}{Figure}{Figures}
\crefname{equation}{}{}
\Crefname{equation}{Eq.\!}{Eqs.\!}

\usepackage[shortlabels]{enumitem}
\setlist[enumerate,2]{label=(\alph*),ref=\theenumi(\alph*)}
\setlist[enumerate,3]{label=\roman*.,ref=\theenumii\roman*}

\usepackage{algorithm,algorithmic}

\usepackage{mathtools}
\let\set\relax
\let\T\relax

\DeclarePairedDelimiterXPP\set[1]{}{\{}{\}}{}{

#1}

\DeclarePairedDelimiter\N{\|}{\|}
\DeclarePairedDelimiter\inner{\langle}{\rangle}

\def\ol{\overline}
\def\ul{\underline}
\def\wtd{\widetilde}
\def\what{\widehat}

\DeclareMathOperator{\diff}{d\!}

\DeclareMathOperator{\rank}{rank}

\DeclareMathOperator{\F}{F}

\DeclareMathOperator{\HH}{H}
\DeclareMathOperator{\T}{T}
\DeclareMathOperator{\OO}{O}

\DeclareMathOperator{\vectorize}{vec}

\newtheorem{theorem}{Theorem}[section]
\newtheorem{lemma}{Lemma}[section]

\theoremstyle{definition}

\newtheorem{remark}{Remark}[section]
\newtheorem{example}{Example}[section]

\numberwithin{equation}{section}
\numberwithin{figure}{section}
\numberwithin{table}{section}
\allowdisplaybreaks

\newcommand\clue[2]{\stackrel{\makebox[0pt][c]{\scriptsize #1}}{#2}\;}

\DeclareMathOperator{\toepL}{\mathcal{L}}
\DeclareMathOperator{\toepU}{\mathcal{U}}
\DeclareMathOperator{\nnz}{nnz}
\DeclareMathOperator{\nres}{NRes}
\DeclareMathOperator{\hami}{\mathcal{H}}
\newcommand\elltwo[1]{\ell_+^{2, #1}}
\usepackage{amsbsy}
\newcommand\bs[1]{\boldsymbol{ #1}}
\usepackage{mathrsfs}
\newcommand\op[1]{\mathscr{ #1}}

\title{The intrinsic Toeplitz structure and its applications in algebraic Riccati equations}

\author{
Zhen-Chen Guo\thanks{Department of Mathematics, Nanjing University,  Nanjing  210093, China; \texttt{e-mail: guozhenchen@nju.edu.cn}. 
Supported in part by NSFC-11901290 and Fundamental Research Funds for the Central Universities.
} \and 
Xin Liang\thanks{Yau Mathematical Sciences Center, Tsinghua University, Beijing 100084, China, and 
Yanqi Lake Beijing Institute of Mathematical Sciences and Applications, Beijing 101408, China;
\texttt{e-mail: liangxinslm@tsinghua.edu.cn}.
Supported in part by NSFC-11901340.
} 
}

\date{}

\begin{document}
\maketitle

\begin{abstract}
	In this paper we derive a Toeplitz-structured closed form of the unique positive semi-definite stabilizing solution for the discrete-time algebraic Riccati equations, especially for the case that the state matrix is not stable.
	Based on the found form and fast Fourier transform,
	we propose a new algorithm for solving both discrete-time and continuous-time large-scale algebraic Riccati equations with low-rank structure.
	It works without unnecessary assumptions, complicated shift selection strategies, or matrix calculations of the cubic order with respect to the problem scale.
	Numerical examples are given to illustrate its features.
	Besides, we show that it is theoretically equivalent to several algorithms existing in the literature in the sense that they all produce the same sequence under the same parameter setting.
\end{abstract}



\section{Introduction}\label{sec:introduction}
Consider a continuous-time algebraic/limiting Riccati equation (CARE)
\begin{equation}\label{eq:CARE:initial}
	A^{\T} X + X A - X BB^{\T} X + C^{\T}C = 0,
\end{equation}
where $A\in \mathbb{R}^{n\times n}, B\in \mathbb{R}^{n\times m}, C\in \mathbb{R}^{l\times n}$.
The CAREs arise in various models related to control theory, such as linear-quadratic optimal regulator design, and $H^2$ and $H^\infty$ controller design for linear systems, see, e.g., \cite{lancasterR1995algebraic, antsaklisM2007linear}.
They also arise in nonlinear systems, like nonlinear controller design by state-dependent Riccati equations \cite{cimen2000statedependent}, or solving differential Riccati equations by implicit integration schemes \cite{dieci1992numerical,bennerM2018numerical}.
Usually \cref{eq:CARE:initial} has infinite many solutions, but in many applications including those mentioned above only the so-called c-stabilizing solution is hoped to be computed.
Here a solution $X$ is called c-stabilizing  if $A-BB^{\T}X$ is stable, namely all the eigenvalues of $A-BB^{\T}X$ lie in the open left half complex plane $\mathbb{C}_-$.
Its existence and uniqueness are guaranteed by the assumption that the pairs $(A,BB^{\T})$ and $(A^{\T},C^{\T}C)$ are c-stabilizable, or equivalently, $\rank(\begin{bmatrix}
	A-\lambda I & BB^{\T}
\end{bmatrix})=\rank(\begin{bmatrix}
A^{\T}-\lambda I & C^{\T}C
\end{bmatrix})=n$ for any $\lambda\in \mathbb{C}\setminus \mathbb{C}_-$.

During many years, people have developed many numerical methods to find out the c-stabilizing solution of \cref{eq:CARE:initial}. 
Reader are referred to \cite{biniIM2012numerical} to obtain an overview.
In this paper, we are focusing on a special case that $A$ is large-scale and sparse, and $B,C$ are low-rank, namely $m,l\ll n$.
The existing methods are categorized into four classes:
\begin{enumerate}
	\item 
projection methods, including extended Krylov subspace method \cite{heyouniJ2009extended}, rational Krylov subspace method \cite{druskinS2011adaptive}, tangential rational Krylov subspace method \cite{druskinSZ2014adaptive}, global extended Krylov subspace method \cite{jbilou2006arnoldi}, etc.;
	\item 
non-projective iterations, including quadratic ADI \cite{wongB2005quadratic}, Cayley transformed Hamiltonian subspace iteration \cite{linS2015new}, RADI \cite{bennerBKS2018radi}, etc.;
	\item 
		Newton-type methods, including the Galerkin projected variant of New\-ton-Kleinman ADI \cite{bennerS2010newtongalerkinadi} and its inexact line-search variant \cite{bennerHSW2016inexact}, etc.;
	\item 
		methods adopted from those suited for small-scale problems, including structure-preserving doubling algorithm (SDA) \cite{chuFL2005structurepreserving, liCLW2013solving}, and Hamiltonian stable subspace methods \cite{amodeiB2010invariant,bennerB2016solution}, etc..
\end{enumerate}
Many more methods and references can be listed if we bring in more details. 
Interested readers are encouraged to look through
a comparison paper \cite{bennerBKS2020numerical} and the references therein.

The methods in the former three classes use a lot of shifts in the calculation process, so a shift selection strategy rather than several pre-chosen shifts is needed.
Different shifts or strategies usually affect the convergence speed significantly.
Moreover, the convergence of those methods usually relies on more assumptions, for example, $A$ is stable. 
To deal with the problems without the guarantee, the preprocessing is necessary and costs not little calculations.
On the opposite, the methods in the latter class, like SDA, only use one shift (or a few shifts if the incorporation technique is adopted),
which helps decrease the calculation that is not directly related to the solution.

On the other hand, the methods in the former three classes only use matrix-vector multiplication and inverse-vector multiplication (that is actually done by linear system solvers),
while SDA uses matrix-matrix and inverse-matrix multiplication (also done by linear system solvers),
which implies that SDA consumes much more time than those in the former classes.

In this paper, first we contribute a Toeplitz-structured closed form of the d-stabilizing solution of discrete-time algebraic Riccati equations (DAREs) by theoretical analysis,
which naturally induces a new algorithm named FFT-based Toeplitz-structured approximation (FTA) to solve DAREs.
The proposed FTA method exploits the fast Fourier transform (FFT) to reduce the time complexity.
Then using a Cayley transformation that transforms CAREs to DAREs, the FTA is successfully adopted to solve CAREs,
where the incorporation technique (a.k.a. defect correction) is applied to deal with the case that the truncated approximation does not provide enough accuracy.
The FTA solves DAREs and CAREs without more assumptions, shift selection strategies, or matrix-matrix/inverse-matrix multiplications.
As a by-product, we show that FTA, SDA, and many other methods like RADI are equivalent under the same parameter setting including the same initial guess $0$ and the same consistent shift, in the sense that they all produce the same sequence (or subsequence).

The rest of the paper is organized as follows. First, some notations are used.
In \cref{sec:preliminary} we present a detailed form of the inverse of special matrices of the form $I+TT^{\T}$ where $T$ is block-Toeplitz,
whose proof, not an easy consequence of the theory on Toeplitz matrices, is put in \cref{ssec:displacement-rank-and-toeplitz-matrix} for readability.
\Cref{sec:dare} generalizes the idea on the Toeplitz operator in the associated discrete-time dynamic systems under good conditions to those without good conditions, and then naturally induces a closed form of the d-stabilizing solution of DAREs, where the special-structured matrices are involved,
which suggests us to develop the FTA method to solve DAREs.
As is shown in \cref{sec:care}, an variant of FTA for CAREs is obtained with the help of Cayley transformation that transforms CAREs to DAREs.
Numerical tests and discussions are given in \cref{sec:experiments-and-discussions}.
Some concluding remarks are provided in \cref{sec:conclusions}.

\textbf{Notation.}
Throughout this paper,
$I_n$ (or simply $I$ if its dimension is clear from the context) is the $n\times n$ identity matrix. 
Given a vector or matrix $X$,
$X^{\T}$, $X^{\HH}$, $\N{X}$, $\N{X}_{\F}$, $\rho(X)$ are its transpose,
conjugate transpose,
spectral norm, Frobenius norm, and spectral radius respectively.
By $X \otimes Y$ denote the Kronecker product of $X$ and $Y$.
By $\Re \alpha$ denote the real part of a complex number $\alpha$.

We use $X\succ 0$ ($X\succeq 0$) to indicate that $X$ is symmetric positive (semi-)definite, and $X\prec 0$ ($X\preceq 0$) if $-X\succ 0$ ($-X\succeq 0$).
Some easy identities are given:
\begin{equation}\label{eq:easy}
	U(I+V^{\T}U)=(I+UV^{\T})U,\qquad
	U(I+V^{\T}U)^{-1}=(I+UV^{\T})^{-1}U.
\end{equation}
Here is the Sherman-Morrison-Woodbury formula: 
\begin{equation}\label{eq:smwf}
(M + UDV^{\T})^{-1} = M^{-1} - M^{-1} U (D^{-1} + V^{\T} M^{-1} U)^{-1} V^{\T} M^{-1}.
\end{equation}
The inverse sign in \cref{eq:easy,eq:smwf} indicates invertibility.
Both will be applied occasionally.  

In addition, all the discussions below are based on the field $\mathbb{R}$.
They are also valid on the field $\mathbb{C}$, with all  $(\cdot)^{\T}$ replaced  by $(\cdot)^{\HH}$. 

\section{Preliminary }\label{sec:preliminary}
The block-Toeplitz matrices appear in the subsequent sections and play an important role in the proposed algorithms. 
Since 1970s, people have known that 
fast and superfast algorithms are valid for Toeplitz matrices, due to its low displacement rank, see, e.g.,~\cite{kailathKM1979displacement,kailathC1989generalized,kailathC1994generalized,kailathS1995displacement,friedlanderMKL1979new}.
However, 
to keep algebraic Riccati equations in mind, here we only introduce the notations related to block-Toeplitz matrices, and give a lemma that is used in the discussions on algebraic Riccati equations,
while its proof is placed in \cref{ssec:displacement-rank-and-toeplitz-matrix}.

Given $A_0,A_1,\dots,A_{m-1}\in \mathbb{R}^{p_1\times p_2}$, we will use
\[
	\toepL_{p_1\times p_2}\left(\begin{bmatrix}
			A_0\\A_1\\\vdots\\A_{m-1}
	\end{bmatrix}\right)
	=\begin{bmatrix}
	   A_0         &        &         &        &    &   \\
	   A_1        & A_0      &         &        &    &   \\
	   A_2       & A_1     & \ddots  &        &    &   \\
	   \vdots    & \ddots       & \ddots  & \ddots &    &   \\
	   \vdots    &        & \ddots        & A_1     & A_0  &   \\
	   A_{m-1} & \cdots & \cdots  & A_2    & A_1 & A_0 \\
   \end{bmatrix}\in \mathbb{R}^{p_1m\times p_2m}.
\]
For ease, $\toepL_{p_1\times p_2}(A)=\toepL_{p_1\times p_2}\left(\begin{bmatrix}
		A_0\\A_1\\\vdots\\A_{m-1}
	\end{bmatrix}\right)$ if $A=\begin{bmatrix}
	A_0\\A_1\\\vdots\\A_{m-1}
\end{bmatrix}$,
and this notation makes no confusion for the subscript $\cdot_{p_1\times p_2}$ demonstrates how the matrix is composed.
Similarly,
\[
	\toepU_{p_1\times p_2}\left(\begin{bmatrix}
			A_0\\A_1\\\vdots\\A_{m-1}
	\end{bmatrix}\right)
	=\begin{bmatrix}
		A_{m-1} & \cdots  & \cdots & A_2    & A_1 & A_0     \\
                & A_{m-1} & \ddots &        & A_2 & A_1     \\
                &         & \ddots & \ddots &     & A_2     \\
                &         & & \ddots & \ddots     & \vdots  \\
				&         & &        & A_{m-1}    & \vdots  \\
                &         &        &        &     & A_{m-1} \\
   \end{bmatrix}\in \mathbb{R}^{p_1m\times p_2m}.
\]
Besides,
\begin{align*}
	\toepL_{p_2\times p_1}\left(\begin{bmatrix}
				A_0&A_1&\cdots&A_{m-1}
		\end{bmatrix}^{\T}\right)^{\T}
		&=
		\toepU_{p_1\times p_2}\left(\begin{bmatrix}
				A_{m-1}\\\vdots\\A_1\\A_0\\
		\end{bmatrix}\right)
		,\\
		\toepU_{p_2\times p_1}\left(\begin{bmatrix}
				A_0&A_1&\cdots&A_{m-1}
		\end{bmatrix}^{\T}\right)^{\T}
		&=
		\toepL_{p_1\times p_2}\left(\begin{bmatrix}
				A_{m-1}\\\vdots\\A_1\\A_0\\
		\end{bmatrix}\right)
		.
\end{align*}

The following lemma will be used several times later.
\begin{lemma}\label{lm:form-of-inverse:dare}
	Given $Y\in \mathbb{R}^{p_1\times p_2}, D_{t-1}\in \mathbb{R}^{p_1(t-1)\times p_2}$, let 
\[
	T_t
	=\toepL_{p_1\times p_2}\left(\begin{bmatrix}
		Y \\ D_{t-1}\\
	\end{bmatrix}\right)
	=\begin{bmatrix}
		Y  & 0\\ D_{t-1} & T_{t-1}
	\end{bmatrix}\in \mathbb{R}^{p_1t\times p_2t}
	.
\]
Then
\begin{subequations}\label{eq:lm:form-of-inverse:dare}
\begin{equation}
	(I_{p_1t}+T_tT^{\T}_t)^{-1}=
			\begin{multlined}[t]
			\toepU_{p_1\times p_1}\left(\begin{bmatrix}
					Q_2\\Q_1\\
					\end{bmatrix}\right)(I_t\otimes Q_1)^{-1}\toepU_{p_1\times p_1}\left(\begin{bmatrix}
					Q_2\\Q_1\\
					\end{bmatrix}\right)^{\T}
			\\+\toepU_{p_1\times p_2}\left(\begin{bmatrix}
					Q_3\\0\\
					\end{bmatrix}\right)(I_t\otimes \left[W+WY^{\T}Y W\right])^{-1}\toepU_{p_1\times p_2}\left(\begin{bmatrix}
					Q_3\\0\\
					\end{bmatrix}\right)^{\T}
	,
			\end{multlined}
\end{equation}
where $Q_1,Q_2,Q_3$ are solutions to the following equations respectively, and $Q_1,W+WY^{\T}YW$ are nonsingular:
\begin{align}
	\left(I_{p_1(t-1)}+D_{t-1}D^{\T}_{t-1}+T_{t-1}T^{\T}_{t-1}\right)Q_3&=D_{t-1},
\quad W=I_{p_1}-Q^{\T}_3D_{t-1},
	\\
	(I_{p_1t}+T_tT^{\T}_t)
	\begin{bmatrix}
		Q_2\\ Q_1
		\end{bmatrix}&=
		\begin{bmatrix}
			0 \\ I_{p_1}\\
	\end{bmatrix},
	\quad Q_1\in \mathbb{R}^{p_1\times p_1}.
\end{align}
\end{subequations}
\end{lemma}

\section{DARE}\label{sec:dare}
Given a linear time-invariant control system in discrete-time:
\begin{equation}\label{eq:d-control-sys}
	\begin{aligned}
		x_0 &\; \text{is given}, \\
		x_{k+1} & = A x_k + B u_k, \qquad k=0,1,2, \dots,\\
			 y_k&=C x_k,
	\end{aligned}
\end{equation}
where $A\in \mathbb{R}^{n\times n}, B\in \mathbb{R}^{n\times m}, C\in \mathbb{R}^{l\times n}$.
Suppose the following condition holds through out this section:
\[
	\fbox{$(A,B)$ is d-stabilizable and $(C,A)$ is d-detectable,}
\]
or equivalently, 
$\rank(\begin{bmatrix}
	A-\lambda I & BB^{\T}
\end{bmatrix})=\rank(\begin{bmatrix}
A^{\T}-\lambda I & C^{\T}C
\end{bmatrix})=n$ for any $\lambda\in \mathbb{C}\setminus \mathbb{D}$,
where $\mathbb{D}$ is the open unit disk.

Its linear-quadratic optimal control can be expressed as
\begin{equation}\label{eq:opt-control}
	\arg\min_{\set{u_k}}\frac{1}{2}\sum_{k=0}^\infty(y_k^{\T}y_k+u_k^{\T}u_k)=
 \set*{u_k = -(I + B^{\T} X_{\star} B)^{-1} B^{\T} X_{\star} A x_k},
\end{equation}
where $X_\star$ is the unique symmetric positive semi-definite  d-stabilizing solution of the DARE~\cite{biniIM2012numerical,chuFLW2004structurepreserving,lancasterR1991solutions,mehrmann1991automomous}:  
\begin{equation}\label{eq:dare}
-X + A^{\T} X (I + BB^{\T}X)^{-1} A + C^{\T}C = 0.  
\end{equation}
Here a solution $X$ is called d-stabilizing, if the closed loop matrix $A_{X}=(I+BB^{\T}X)^{-1}A$ is d-stable,
namely all of its eigenvalues lie in the open unit disk $\mathbb{D}$,
or equivalently, $\rho(A_X)<1$.

\subsection{In the operator view}\label{ssec:in-the-view-of-operator-theory}
In this section, we briefly state the existence and uniqueness of $X_\star$ shown by the operator theory, which is based on the monograph \cite{ionescuOW1999generalized}.

In order to make things simple, first we assume that $A$ is d-stable.
Write $\bs x=\set{x_k}_{k\in \mathbb{N}},
\bs u=\set{u_k}_{k\in \mathbb{N}},
\bs y=\set{y_k}_{k\in \mathbb{N}}
$.
Let $\elltwo n$ denote the Hilbert space of norm-square summable $\mathbb{R}^n$-valued series.

Suppose $\bs u\in \elltwo m$ and consider the cost functional (a.k.a. restricted quadratic index)
\[
	J(\bs u)=\sum_{k=0}^{+\infty}\begin{bmatrix}
		x_k \\ u_k
	\end{bmatrix}^{\T}\begin{bmatrix}
	Q & L \\ L^{\T} & R
	\end{bmatrix}\begin{bmatrix}
		x_k \\ u_k
	\end{bmatrix}=\inner*{\begin{bmatrix}
		\bs x\\ \bs u
	\end{bmatrix}, \begin{bmatrix}
	Q & L \\ L^{\T} & R
	\end{bmatrix}\begin{bmatrix}
		\bs x\\ \bs u
\end{bmatrix}}_{\elltwo n\times \elltwo m},
\]
where $\bs x$ satisfies \cref{eq:d-control-sys}.
Here for any matrix $U$ and any series $\bs z=\set{z_k}_{k\in \mathbb{N}}$, $U\bs z$ is understood as $U\bs z:=\set{Uz_k}_{k\in \mathbb{N}}$.

In fact $\bs x=\op F x_0+\op L \bs u\in \elltwo n$,
where $\op F\colon \mathbb{R}^n \to \elltwo n, (\op F x_0)_k=A^k x_0,k\ge 0$, and
$\op L\colon \elltwo m \to \elltwo n, (\op L \bs u)_0=0, (\op L \bs u)_k=\sum_{i=0}^{k-1}A^{k-i-1}Bu_i, k\ge 1$.
Clearly $\op F$ and $\op L$ are bounded linear operators.
Also, it is not difficult to find $\op L$ is a Toeplitz operator.
Hence
\begin{align*}
	J(\bs u)
	&=\inner*{\begin{bmatrix}
		\bs x\\ \bs u
	\end{bmatrix}, \begin{bmatrix}
	Q & L \\ L^{\T} & R
	\end{bmatrix}\begin{bmatrix}
		\bs x\\ \bs u
\end{bmatrix}}
\\&=\inner*{\begin{bmatrix}
	\op F & \op L\\ & I
	\end{bmatrix}\begin{bmatrix}
		x_0 \\ \bs u
\end{bmatrix}, \begin{bmatrix}
	Q & L \\ L^{\T} & R
	\end{bmatrix}\begin{bmatrix}
	\op F & \op L\\ & I
	\end{bmatrix}\begin{bmatrix}
		x_0 \\ \bs u
\end{bmatrix}}
\\&=\inner*{\begin{bmatrix}
		x_0 \\ \bs u
\end{bmatrix}, \begin{bmatrix}
	\op P_o & \op P\\ \op P^*& \op R
	\end{bmatrix}\begin{bmatrix}
		x_0 \\ \bs u
\end{bmatrix}}\qquad (\text{$\op A^*$ is the adjoint of operator $\op A$})
\end{align*}
where $\op P_o=\op F^* Q\op F, \op P=\op F^*(Q\op L+L), \op R= R+L^{\T}\op L+\op L^* L + \op L^* Q\op L$.
Then the unique symmetric d-stabilizing solution $X_{\star}$ of the DARE \cref{eq:dare} is given by
\begin{equation}\label{eq:closed-sol:operator}
	X_{\star}=\op P_o-\op P\op R^{-1}\op P^*.
\end{equation}
Clearly, $\op P_o, \op P, \op R$ are bounded linear operators.
\cite[Theorem~4.7.1]{ionescuOW1999generalized} tells that the DARE \cref{eq:dare} has a unique d-stabilizing solution, if and only if the Toeplitz-like operator $\op R$ has a bounded inverse.


In the following, we derive the (infinite) matrix representation of \cref{eq:closed-sol:operator}, 
which is not provided in \cite{ionescuOW1999generalized}.
Here we only show a simple case that $Q=C^{\T}C,L=0,R=I$.
Obviously, the matrix representations of $\op F, \op L$, still denoted by $\op F,\op L$, are
\[
	\op F=\begin{bmatrix}
		I\\ A\\ A^2 \\A^3\\\vdots
	\end{bmatrix},\qquad
	\op L=\begin{bmatrix}
		0 \\
		B & 0\\
		AB & B & 0\\
		A^2B & AB & B & 0 \\
		\vdots & \ddots & \ddots & \ddots & \ddots 
	\end{bmatrix}.
\]
Hence
\begin{align*}
	X_{\star}
	&={\op F}^* Q{\op F}-{\op F}^*( Q{\op L}+ L)( R+ L^{\T}{\op L}+{\op L}^*  L + {\op L}^*  Q{\op L})^{-1}({\op L}^* Q^{\T}+ L^{\T}){\op F}
	\\&={\op F}^*  C^{\T}C{\op F}-{\op F}^* C^{\T}C{\op L}( I+ {\op L}^*  C^{\T}C{\op L})^{-1}{\op L}^* C^{\T}C{\op F}
	\\&={\op F}^*  C^{\T}\left[I-C{\op L}( I+ {\op L}^*  C^{\T}C{\op L})^{-1}{\op L}^* C^{\T}\right]C{\op F}
	\\&\clue{\cref{eq:smwf}}{=} {\op F}^*  C^{\T}( I+ C{\op L}{\op L}^* C^{\T})^{-1}C{\op F}
	.
\end{align*}
Write 
\begin{equation}\label{eq:noniter:X:UVT:inf}
	\op V = C\op F=\begin{bmatrix}
		C \\ CA \\ CA^2\\ CA^3\\ \vdots 
	\end{bmatrix} 
	,\quad
	\op T = C\op L=\begin{bmatrix}
		0 \\
		CB & 0\\
		CAB & CB & 0\\
		CA^2B & CAB & CB & 0 \\
		\vdots & \ddots & \ddots & \ddots & \ddots 
	\end{bmatrix}
	,
\end{equation}
and then $\op V\colon \mathbb{R}^n\to\elltwo l,\op T\colon\elltwo m\to\elltwo l, ( I+ {\op T}{\op T}^*)^{-1}\colon \elltwo l\to\elltwo l$ are bounded linear operators, and $\op T$ is also a Toeplitz operator.
Thus
\begin{equation}\label{eq:noniter:X:X:inf}
	X_{\star}
	=\op V^*(I+\op T\op T^*)^{-1}\op V,
\end{equation}
which is a closed form of the d-stabilizing solution.

For the case that $A$ is not d-stable, does a similar closed form of the d-stabilizing solution exist?
In this case, the operators $\op F,\op L$ are no longer bounded linear operators and the series involved may not converge.
In the next subsection, we will show \cref{eq:noniter:X:X:inf} is also the d-stabilizing solution of the DARE \cref{eq:dare} even for the unstable case.

\subsection{In the matrix view}\label{ssec:in-the-view-of-matrix-analysis}

It is well known that $X_\star=\lim_{t\to\infty}X_t$, where $X_t$ is generated by the  difference Riccati equation (DRE):
\begin{equation}\label{eq:ddre}
	X_0=0,\qquad X_{t+1}=\op D(X_t):=C^{\T}C+A^{\T}X_t(I+BB^{\T}X_t)^{-1}A,
\end{equation}
which can be recognized as a variant of fixed point iteration for \cref{eq:dare}.

Based on the fixed point iteration, 
in 1970s, people have developed the doubling algorithm to solve DAREs~\cref{eq:dare} and CAREs. 
Anderson \cite{anderson1978secondorder} proposed a variant, which is recently usually called SDA and has three iterative recursions: 
\begin{subequations}\label{eq:sda}
	\begin{align}
	A_{k+1}& = A_k (I_n + G_k H_k)^{-1} A_k, \label{eq:sda:A}\\ 
	G_{k+1} &= G_k + A_k (I_n + G_k H_k)^{-1}  G_k A_k^{\T},\label{eq:sda:G} \\ 
	H_{k+1} &= H_k + A_k^{\T} H_k(I_n + G_k H_k)^{-1} A_k, \label{eq:sda:H}
	\end{align}
\end{subequations}
provided that all matrix inversions are feasible (i.e.,  $I_n + G_k H_k$ are nonsingular for $k=0, 1, \cdots$).
The initial terms are usually set by
\begin{equation*}\label{eq:initial:dare}
A_0=A,\qquad    
G_0=BB^{\T},\qquad 	
H_0=C^{\T}C.
\end{equation*}
It has been shown that for those initial terms, $I_n + G_k H_k$ are nonsingular for $k\ge 0$, and 
$A_k \to  0$, $G_k \to  Y_\star$ (the solution to the dual DARE) and $H_k \to  X_\star$, all quadratically~\cite{mehrmann1991automomous} except for the critical case~\cite{huangL2009structured}.  

In \cite{anderson1978secondorder}, it is stated clearly that $H_k=X_{2^k}$, implying that the iteration for $H_k$ can be treated as an acceleration of \cref{eq:ddre},
because it only computes the terms $X_1,X_2,X_4,\dots,X_{2^k},\dots$.
Moreover, \cite{anderson1978secondorder} also argued that \cref{eq:ddre} with any initial $X_0\succeq0$ leads $X_t\to X_\star$ in usual situation (but did not mention which situation satisfies).

Questions arise naturally, of which two are:
\begin{enumerate}
	\item can we even only compute less terms in the sequence $\set{X_t}$, namely accelerate \cref{eq:ddre} even further?
		\item how things go when arbitrary initial terms are set?
\end{enumerate}

Before we begin the analysis, a simple property of the DRE \cref{eq:ddre} is given.
\begin{lemma}\label{lm:ddre:monotonic}
	The operator $\op D$ is monotonic on the set consisting of all positive semi-definite matrices with respect to the partial order ``$\,\succeq$''.
	In details, if $Z_1\succeq0,Z_2\succeq0$, then
	\[
		Z_1\succeq Z_2 \Rightarrow \op D(Z_1)\succeq\op D(Z_2).
	\]
\end{lemma}
\begin{proof}
	First suppose $Z_2\succ0$ and thus $Z_2$ is nonsingular.
	Then
	\begin{align*}
		Z_1\succeq Z_2
		\Leftrightarrow Z_1^{-1}\preceq Z_2^{-1}
		&\Leftrightarrow (Z_1^{-1}+BB^{\T})^{-1}\succeq (Z_2^{-1}+BB^{\T})^{-1}
		\\&\Leftrightarrow Z_1(I+BB^{\T}Z_1)^{-1}\succeq Z_2(I+BB^{\T}Z_2)^{-1}
		\Rightarrow \op D(Z_1)\succeq\op D(Z_2)
		.
	\end{align*}
	If $Z_2$ is singular, then $Z_2+\varepsilon I\succ0$ for any $\varepsilon>0$.
	Thus, taking limits yields
	\[
		Z_1\succeq Z_2
		\Leftrightarrow Z_1+\varepsilon I \succeq Z_2+\varepsilon I
		\Rightarrow \op D(Z_1+\varepsilon I)\succeq\op D(Z_2 + \varepsilon I)
		\Rightarrow \op D(Z_1)\succeq\op D(Z_2)
		.
		\qedhere
	\]
\end{proof}
Then we may ensure the validity of \cref{eq:noniter:X:X:inf}.
\begin{theorem}\label{thm:fixedpoint}
Write 
	\begin{equation}\label{eq:noniter:X:UVT}
			V_t = \begin{bmatrix}
					C \\ CA \\ CA^2\\ \vdots \\ \vdots \\ CA^{t-1}
				\end{bmatrix} 
				,\quad
				T_t = \begin{bmatrix}
				   0         &        &         &        &    &   \\
				   CB        & 0      &         &        &    &   \\
				   CAB       & CB     & \ddots  &        &    &   \\
				   \vdots    &        & \ddots  & \ddots &    &   \\
				   \vdots    &        &         & CB     & 0  &   \\
				   CA^{t-2}B & \cdots & \cdots  & CAB    & CB & 0 \\
				\end{bmatrix}
				,
	\end{equation}
	$T_1=0$. 
	Then the terms of the sequence $\set{X_t}$ generated by the DRE \cref{eq:ddre} are given by
	\begin{equation}\label{eq:noniter:X:X}
		X_t
		= V_t^{\T}(I+T_tT_t^{\T})^{-1}V_t
		, \qquad t=1,2,\dots.
	\end{equation}
	Moreover, $\set{X_t}$ is monotonically nondecreasing, and $X_t\to X_{\star}$, the d-stabilizing solution of DARE \cref{eq:dare}.
\end{theorem}
\begin{proof}
Clearly $X_0=0,X_1=C^{\T}C$.
Using some calculations, we have
\begin{align*}
	X_2
	&=C^{\T}C+A^{\T}C^{\T}C(I+BB^{\T}C^{\T}C)^{-1}A
	\\&=C^{\T}C+A^{\T}C^{\T}(I+CBB^{\T}C^{\T})^{-1}CA
	\\&=\begin{bmatrix}
		C\\ CA
		\end{bmatrix}^{\T}\begin{bmatrix}
	I & \\ &I+CBB^{\T}C^{\T}
	\end{bmatrix}^{-1}\begin{bmatrix}
		C\\ CA
	\end{bmatrix}
	\\&=\begin{bmatrix}
		C\\ CA
		\end{bmatrix}^{\T}\left(I+\begin{bmatrix}
	0 & \\ CB& 0
	\end{bmatrix}\begin{bmatrix}
	0 & \\ CB& 0
	\end{bmatrix}^{\T}\right)^{-1}\begin{bmatrix}
		C\\ CA
	\end{bmatrix}
	.
\end{align*}

	Now \cref{eq:noniter:X:X} is correct for $t=1,2$.
	Assuming \cref{eq:noniter:X:X} is correct for $t$, we are going to prove it is also correct for $t+1$.
	By the DRE \cref{eq:ddre},
\begin{align*}
	X_{t+1}
	&= C^{\T}C + A^{\T}V_t^{\T}\left(I+T_tT_t^{\T}\right)^{-1}V_t
	\left(I+BB^{\T}V_t^{\T}\left(I+T_tT_t^{\T}\right)^{-1}V_t\right)^{-1}A
	\\&\clue{\cref{eq:easy}}{=}
	C^{\T}C + A^{\T}V_t^{\T}\left(I+T_tT_t^{\T}\right)^{-1}
	\left(I+V_tBB^{\T}V_t^{\T}\left(I+T_tT_t^{\T}\right)^{-1}\right)^{-1}V_tA
	\\&=
	C^{\T}C + A^{\T}V_t^{\T}
	\left(I+T_tT_t^{\T}+V_tBB^{\T}V_t^{\T}\right)^{-1}V_tA
	\\&=\begin{bmatrix}
		C\\ V_tA
		\end{bmatrix}^{\T}\begin{bmatrix}
	I & \\ &I+T_tT_t^{\T}+V_tBB^{\T}V_t^{\T}
	\end{bmatrix}^{-1}\begin{bmatrix}
		C\\ V_tA
	\end{bmatrix}
	\\&=\begin{bmatrix}
		C\\ V_tA
		\end{bmatrix}^{\T}\left(I+\begin{bmatrix}
	0 & \\ V_tB& T_t
	\end{bmatrix}\begin{bmatrix}
	0 & \\ V_tB& T_t
	\end{bmatrix}^{\T}\right)^{-1}\begin{bmatrix}
		C\\ V_tA
	\end{bmatrix}
	\\&=V_{t+1}^{\T}(I+T_{t+1}T_{t+1}^{\T})^{-1}V_{t+1}
	.
	\qedhere
\end{align*}

	Then we illustrate the monotonicity of the sequence. 

Since $X_1\succeq X_0=0$, by \cref{lm:ddre:monotonic}, $X_2=\op D(X_1)\succeq\op D(X_0)=X_1$.
Similarly $0=X_0\preceq X_1\preceq X_2\preceq\dots\preceq X_t\preceq\cdots$, namely the sequence $\set{X_t}$ generated by \cref{eq:ddre} is monotonic.
On the other hand, the d-stabilizing solution $X_\star$ is also the unique symmetric positive semi-definite solution. Thus, $X_\star\succeq 0=X_0$, and $X_\star=\op D(X_\star)\succeq \op D(X_0)=X_1$.
Similarly $X_\star\succeq X_t$ for any $t$, namely the sequence $\set{X_t}$ is bounded.
As a result, $X_t$ converges. Write $X_t\to X_\infty$, and then $X_\infty$ is a symmetric positive semi-definite solution to the DARE \cref{eq:dare}.
Then the uniqueness of the symmetric positive semi-definite solution forces $X_\star=X_\infty$.
In other words, it holds that $X_t\to X_{\star}$.
\end{proof}
It is not difficult to discover that \cref{thm:fixedpoint} coincides with the decoupled formulae of the dSDA for DAREs introduced in \cite{guoCLL2020decoupled} at $t=2^k$,
which is actually guaranteed by the fact that the sequence $\set{H_k}$ generated by SDA \cref{eq:sda:H} is a subsequence of $\set{X_t}$.

One can easily find \cref{eq:noniter:X:X} is the truncated form of \cref{eq:noniter:X:X:inf}, a Toeplitz-structured closed form of $X_\star$, whose validity for the d-stable case has been proved by the operator theory in \cref{ssec:in-the-view-of-operator-theory}.
Note that under the assumption that $A$ is d-stable, $V_t$ and $T_t$, treated as the truncations of $\op V$ and $\op T$, converges to $\op V$ and $\op T$ respectively, by the fact that $\op V,\op T$ are bounded linear operators. With the help of operator theory, $X_t\to X_\star$.
To the opposite, for the case that $A$ is not d-stable, $\op V$ and $\op T$ are no longer bounded, and it would be difficult to show $X_t\to X_\star$ by the operator theory. 
However, the matrix analysis reveals that $X_t\to X_\star$, which implies $X_\star$ indeed has the closed form \cref{eq:noniter:X:X:inf} in the unstable case.

\subsection{Efficient method}\label{ssec:efficient-method}
Now we acquire the non-iterative form \cref{eq:noniter:X:X} of $X_t$,
which allows us to compute the terms $X_t$ directly for arbitrary $t$.
In the following, we will work on an efficient method to compute $X_t$ for any given $t$.

Using the notations for Toeplitz matrices in \cref{sec:preliminary},
we have 
\[
	T_t
	=\toepL_{l\times m}\left(\begin{bmatrix}
		0\\ V_{t-1}B\\
	\end{bmatrix}\right)
	=\begin{bmatrix}
		0 & 0\\
		\toepL_{l\times m}( V_{t-1}B) & 0 \\
	\end{bmatrix}=\begin{bmatrix}
		0 & 0\\ V_{t-1}B & T_{t-1}
	\end{bmatrix}
	.
\]
Thus,
\[
\toepL_{l\times m}( V_{t-1}B)\toepL_{l\times m}( V_{t-1}B)^{\T}
=T_{t-1}T_{t-1}^{\T}+V_{t-1}BB^{\T}V_{t-1}^{\T},\
\]
and 
\begin{equation}\label{eq:Xt=CTC+A}
		X_t=
		C^{\T}C + A^{\T}V_{t-1}^{\T}
			\left[I+\toepL_{l\times m}( V_{t-1}B)\toepL_{l\times m}( V_{t-1}B)^{\T}\right]^{-1}V_{t-1}A
			.
\end{equation}
Clearly $\toepL_{l\times m}( V_{t-1}B)$ is block-Toeplitz.
Hence the results in \cref{sec:preliminary} can be applied.


By \cref{lm:form-of-inverse:dare}, $X_t$ can be computed by solving only $l+m$ rather than $n$ equations, shown in \cref{thm:form-of-solution:dare}.
\begin{theorem}\label{thm:form-of-solution:dare}
Let 
\begin{subequations}\label{eq:Q}
	\begin{align}
		\left(I_{(t-1)l}+\toepL_{l\times m}(V_{t-1}B)\toepL_{l\times m}(V_{t-1}B)^{\T}\right)\begin{bmatrix}
			Q_2\\Q_1\\
			\end{bmatrix}&=
			\begin{bmatrix}
				0 \\ I_l \\
			\end{bmatrix}
			,\qquad Q_1\in \mathbb{R}^{l\times l}
			,\\
			\left(I_{(t-1)l}+\toepL_{l\times m}(V_{t-1}B)\toepL_{l\times m}(V_{t-1}B)^{\T}\right)\begin{bmatrix}
				Q_4\\Q_3\\
				\end{bmatrix}&=
				 V_{t-1}B
			,\qquad Q_4\in \mathbb{R}^{l\times m}
			,
	\end{align}
\end{subequations}
and $W=I_{m}-\begin{bmatrix}
				Q_4\\Q_3\\
			\end{bmatrix}^{\T}
			V_{t-1}B$.
	Then the sequence $X_t$ defined by \cref{eq:noniter:X:X} can be generated by
	\begin{equation}\label{eq:Xt=SDS}
		X_t= 
			\begin{bmatrix}
				C \\ \Xi_1 \\  \Xi_2
			\end{bmatrix}^{\T}\begin{bmatrix}
			I_{l} &&\\& (I_{t-1}\otimes Q_1)^{-1}&\\&&(I_{t-1}\otimes W)^{-1}
			\end{bmatrix}\begin{bmatrix}
				C \\ \Xi_1 \\  \Xi_2
			\end{bmatrix},
	\end{equation}
	where
	\[
		\Xi_1 = \toepU_{l\times l}\left(\begin{bmatrix}
				Q_2\\Q_1\\
		\end{bmatrix}\right)^{\T}V_{t-1}A \in \mathbb{R}^{(t-1)l\times n},
		\qquad
		\Xi_2=\toepU_{l\times m}\left(\begin{bmatrix}
				Q_3\\0\\
		\end{bmatrix}\right)^{\T} V_{t-1}A\in \mathbb{R}^{(t-1)m\times n}
		.
	\]
\end{theorem}
\begin{proof}
	By \cref{lm:form-of-inverse:dare} with $Y\leftarrow 0, D_{t-1}\leftarrow V_{t-1}B$, we have
\[
	(I+T_tT_t^{\T})^{-1}=
	\toepU_{l\times l}\left(\begin{bmatrix}
			\wtd Q_2\\Q_1\\
			\end{bmatrix}\right)(I\otimes Q_1)^{-1}\toepU_{l\times l}\left(\begin{bmatrix}
			\wtd Q_2\\Q_1\\
	\end{bmatrix}\right)^{\T}
	+\toepU_{l\times m}\left(\begin{bmatrix}
			\wtd Q_3\\0\\
			\end{bmatrix}\right)(I\otimes W)^{-1}\toepU_{l\times m}\left(\begin{bmatrix}
			\wtd Q_3\\0\\
	\end{bmatrix}\right)^{\T}
	,
\]
where
\begin{align*}
	\left(I+\toepL_{l\times m}(V_{t-1}B)\toepL_{l\times m}(V_{t-1}B)^{\T}\right)\wtd Q_3&=V_{t-1}B,
	\quad W=I-\wtd Q_3^{\T}V_{t-1}B,
	\\
	(I+T_tT_t^{\T})
	\begin{bmatrix}
		\wtd Q_2\\ Q_1
		\end{bmatrix}&=
		\begin{bmatrix}
		0 \\ I_l\\
	\end{bmatrix},
	\qquad Q_1\in \mathbb{R}^{l\times l}.
\end{align*}
Note that 
\[
(I+T_tT_t^{\T})^{-1}=\begin{bmatrix}
	I & \\ & 
	\left(I+\toepL_{l\times m}(V_{t-1}B)\toepL_{l\times m}(V_{t-1}B)^{\T}\right)^{-1}
\end{bmatrix}
.
\]
Hence $\wtd Q_2=\begin{bmatrix}
	0_{l\times l} \\ Q_2\\
\end{bmatrix}$ where
\[
	\left(I+\toepL_{l\times m}(V_{t-1}B)\toepL_{l\times m}(V_{t-1}B)^{\T}\right)
\begin{bmatrix}
	Q_2\\ Q_1
\end{bmatrix}=
		\begin{bmatrix}
		0 \\ I_l\\
	\end{bmatrix}.
\]
Write $\wtd Q_3=\begin{bmatrix}
	Q_4\\ Q_3
\end{bmatrix}$ where $Q_4\in \mathbb{R}^{l\times m}$,
and then it follows that
\begin{equation}\label{eq:I+LLT:inv}
	\begin{multlined}
		\left(I+\toepL_{l\times m}(V_{t-1}B)\toepL_{l\times m}(V_{t-1}B)^{\T}\right)^{-1}
		=
		\toepU_{l\times l}\left(\begin{bmatrix}
				Q_2\\Q_1\\
				\end{bmatrix}\right)(I\otimes Q_1)^{-1}\toepU_{l\times l}\left(\begin{bmatrix}
				Q_2\\Q_1\\
		\end{bmatrix}\right)^{\T}
		\\+\toepU_{l\times m}\left(\begin{bmatrix}
				Q_3\\0\\
				\end{bmatrix}\right)(I\otimes W)^{-1}\toepU_{l\times m}\left(\begin{bmatrix}
				Q_3\\0\\
		\end{bmatrix}\right)^{\T}
		.
	\end{multlined}
\end{equation}
Then the result is a direct consequence of \cref{eq:Xt=CTC+A}.
\end{proof}
\Cref{thm:form-of-solution:dare} suggests a new algorithm, \cref{alg:fna-d}, to approximate the solution of DAREs.
The key is how to fast compute $Q_{*,*}$, or equivalently solve the linear systems \cref{eq:Q}, and compute the products of $\toepU_{*}(*)^{\T}V_{t-1}A$.
Both are related to the manipulations on block Toeplitz matrices.
It is well known that the fast Fourier transform (FFT) can be used to accelerate the calculation with Toeplitz matrices involved, see, e.g., \cite{strang1986proposal,jin2002developments,jin2010preconditioning} and the references therein.

\begin{algorithm}[h]
	\caption{FFT-based Toeplitz-structured Approximation (FTA) for DAREs}\label{alg:fna-d}
	\begin{algorithmic}[1]
		\REQUIRE $A\in \mathbb{R}^{n\times n},B\in \mathbb{R}^{n\times m},C\in \mathbb{R}^{l\times n}$ and $t$.
		\STATE Compute sequentially 
		$C\cdot A, CA\cdot A,\dots,CA^{t-3} \cdot A, CA^{t-2}\cdot A$, and form $V_{t-1}\in \mathbb{R}^{(t-1)l\times n}$ by stacking $C$ and the first $t-2$ terms vertically in order, and form $V_{t-1}A$ by stacking the $t-1$ terms vertically in order.
		\STATE Compute $V_{t-1}B\in \mathbb{R}^{(t-1)l\times m}$.
		\STATE Use Preconditioned Conjugate Gradient (PCG) method to solve \cref{eq:Q}.
		\STATE Compute $W=I_{m}-\begin{bmatrix}
				Q_4\\Q_3\\
			\end{bmatrix}^{\T}
			 V_{t-1}B$ and then the Cholesky factorizations of $Q_1=L_QL_Q^{\T}$ and $W=L_WL_W^{\T}$.
		\STATE Use fast multiplication to obtain $S_1=\toepU_{l\times l}\left(\begin{bmatrix}
				Q_2\\Q_1\\
				\end{bmatrix}L_Q^{-\T}\right)^{\T}V_{t-1}A\in \mathbb{R}^{(t-1)l\times n}$ and $S_2=\toepU_{l\times m}\left(\begin{bmatrix}
				Q_3\\0\\
		\end{bmatrix}L_W^{-\T}\right)^{\T} V_{t-1}A\in \mathbb{R}^{(t-1)m\times n}$,
		and form $S=\begin{bmatrix}
			C \\ S_1\\ S_2
		\end{bmatrix}\in \mathbb{R}^{(tl+(t-1)m)\times n}$.
		\ENSURE $S $ which satisfies $S ^{\T}S \approx X_\star\in \mathbb{R}^{n\times n}$.
	\end{algorithmic}
\end{algorithm}

Some remarks are given below to illustrate the algorithm.
\paragraph{Parameter and output}

\begin{enumerate}
	\item In order to use FFT, $t$ is usually chosen as powers of $2$, namely $t=2^k$.
		There is no strategy to determine a proper $t$ in advance. In practice, we may choose a heuristic $k$, for example $5$--$8$.
		If the output is a good approximation of the solution, then we stop here; otherwise, we use the output as a new initial guess, and run another round to achieve a better approximation; the process is repeated until convergence, namely some criterion is satisfied.
		The details and the validity of implementing a new initial guess are discussed in \cref{ssec:arb-term}.
	\item
		Note that $\rank(S)=\rank(V_{2^k})$. 
		Numerically $V_{2^k}$ probably has rank much less than $2^kl$.
		An obvious clue is that $V_{2^k}$ contains a power series of $A$ performing on $C$, and as $k$ goes larger and larger, the terms in it become more and more likely to be linearly dependent.
		This implies that \cref{eq:Xt=SDS} is not a compact form.
		To deal with this, some compression technique may be brought in.
		This idea needs more discussions on the convergence, which is given in \cref{ssec:arb-term}.
		\item 
			In step~3, we use PCG to solve \cref{eq:Q}. To make calculation least, the preconditioner can be chosen as the diagonal part of the linear system. We will use it in the Experiments part below. Other preconditioners may also be considered. In practice, the number of steps of PCG is fixed on an integer $M$. One reason is that if the condition number of the system is not too large, then the PCG would converge fast; another reason is that stopping in the midway will not hurt the outer convergence on $X_\star$, which is implied by \cref{lm:ddre:attractor} below. 
	\item 
		In the output, we do not give an approximation of $X_\star$ directly but its factor, namely 
		a $tl\times n$ matrix $S$.
		If some compression technique is used during the process, an approximation of $S$ would have relatively small low row rank, say $r\ll n$.
		Then in practice we only need the products of $X$ and other matrices, for example, in obtaining the optimal control \cref{eq:opt-control}.
		The setting $r<n$ makes multiplication with $X_\star$'s factor save time and space.
		This is also considered in many literatures, e.g.,~\cite{bennerBKS2020numerical}.
\end{enumerate}

\paragraph{Time complexity}
Complexity for $S$, the factor of $X_t$:
\begin{enumerate}
	\item Step~1, compute $V_{t-1}A$, namely $CA,CA^2,\dots,CA^{t-1}$, in $(t-1)(2n-1)nl$ flops. 
	\item Step~2, compute $V_{t-1}B$ in $(t-1)(2n-1)lm$ flops.
	\item Step~3, use $M$-step PCG (suppose one-step PCG is done in $Ntl\ln l$ flops for fast multiplication where $N$ is a constant), to compute $Q_{*,*}$, in $\OO(MNtl[\ln^2t+\ln(tl)](l+m))$ flops.
	\item Step~4, compute $W$ in $(2tl-1)\frac{m(m+1)}{2}+m$ flops, and $L_Q,L_W,\wtd D_Q,\wtd D_W$ in $\frac{2}{3}l(l-1)(l+4)$ flops.
	\item Step~5, compute $\begin{bmatrix}
				Q_2\\Q_1\\
				\end{bmatrix}L_Q^{-\T},\begin{bmatrix}
				Q_3\\0\\
			\end{bmatrix}L_W^{-\T}$ in $tl^3+(t-1)lm^2$ flops; compute $S_1,S_2$ in $\OO(2Nntl\ln l)$ flops.
	\item To sum up, assuming $l\ll n, m\ll n$ and omitting lower order terms, the total complexity is $2ln^2t+2lmnt+\OO(MNl(l+m)t\ln^2t+2Nltn\ln l)=\OO\left(t(n^2+\ln^2t)\right)$ flops.
	\item Suppose $A$ is sparse, and the number of nonzero entries is $\nnz(A)$.
		Only Step~1 is different, and the total complexity is $\OO(lt\nnz(A)+lmnt+MNl(l+m)t\ln^2t+2Nltn\ln l)=\OO\left(t(\nnz(A)+n+\ln^2t)\right)$ flops.
\end{enumerate}
\paragraph{Space complexity}
\begin{enumerate}
	\item Step~1, store $V_{t-1}A$ in $(t-1)nl$ units.
	\item Step~2, store $V_{t-1}B$ in $(t-1)lm$ units.
	\item Step~3, store $Q_{*,*}$ in $(t-1)l(l+m)$ units.
	\item Step~4, store $W$ and then $L_W,L_Q$ in part of the storage for $V_{t-1}B$. (The storage is enough and no extra units are needed because the three matrices need in total $m(m+1)/2+l(l+1)/2$ units, which is less than $(t-1)lm$.)
	\item Step~5, store $S_2$ in the storage for $V_{t-1}B$ and additional storage, consuming in total $(t-1)ln$ units; 
		store $S_1$ in the storage for $V_{t-1}A$. 
	\item to sum up, the total storage is $(t-1)l(2n+l+m)=\OO(tn)$ units.
\end{enumerate}

\subsection{Arbitrary initial term}\label{ssec:arb-term}
In this subsection, we consider \cref{eq:ddre} with an arbitrary initial $X_0=\Gamma^{\T}\Gamma\succeq0$ with $\Gamma\in \mathbb{R}^{\wtd l\times n}$: 
\begin{equation}\label{eq:ddre:arb}
	X_0=\Gamma^{\T}\Gamma,\qquad X_{t+1}=\op D(X_t)=C^{\T}C+A^{\T}X_t(I+BB^{\T}X_t)^{-1}A.
\end{equation}
Note that \cref{eq:ddre:arb} is the same iteration as \cref{eq:ddre} with a different initial matrix.
\begin{theorem}\label{thm:dare:arb}
	Write
	\[
		\Upsilon_t = \begin{bmatrix}
			A^{t-1}B&\cdots &AB&B
		\end{bmatrix}.
	\]
	Then the sequence $\set{X_t}$ generated by \cref{eq:ddre:arb} is given by
	\begin{equation}\label{eq:noniter:X:arb}
		X_t=\begin{bmatrix}
			V_t\\ \Gamma A^t
			\end{bmatrix}^{\T}\left(I+\begin{bmatrix}
				T_t \\ \Gamma \Upsilon_t
		\end{bmatrix}\begin{bmatrix}
				T_t \\ \Gamma \Upsilon_t
	\end{bmatrix}^{\T}\right)^{-1}\begin{bmatrix}
			V_t\\ \Gamma A^t
		\end{bmatrix}
		,
	\end{equation}
	where $V_t,T_t$ is defined by \cref{eq:noniter:X:UVT}.
\end{theorem}
\begin{remark}\label{rk:thm:form-of-solution:dare:arb}
	Note that \cref{eq:noniter:X:arb} coincides with \cref{eq:noniter:X:X} at $\Gamma=0$.
\end{remark}
\begin{proof}
	First examine $X_1$.
	\begin{align*}
		\begin{bmatrix}
			C\\ \Gamma A
			\end{bmatrix}^{\T}\left(I+\begin{bmatrix}
				0 \\ \Gamma B 
				\end{bmatrix}\begin{bmatrix}
				0 \\ \Gamma B
			\end{bmatrix}^{\T}\right)^{-1}\begin{bmatrix}
			C\\ \Gamma A
		\end{bmatrix}
		&=C^{\T}C+A^{\T}\Gamma^{\T}(I+\Gamma BB^{\T}\Gamma^{\T})^{-1}\Gamma A
		\\&=C^{\T}C+A^{\T}\Gamma^{\T}\Gamma(I+ BB^{\T}\Gamma^{\T}\Gamma)^{-1} A
		=X_1.
	\end{align*}
	Then examine the recursion.
	\begin{align*}
		\MoveEqLeft[1]
		C^{\T}C+A^{\T}X_t(I+BB^{\T}X_t)^{-1}A
		\\&=
		C^{\T}C+A^{\T}\begin{bmatrix}
			V_t\\ \Gamma A^t
			\end{bmatrix}^{\T}\left(I+\begin{bmatrix}
				T_t \\ \Gamma \Upsilon_t 
		\end{bmatrix}\begin{bmatrix}
				T_t \\ \Gamma \Upsilon_t 
	\end{bmatrix}^{\T}\right)^{-1}\begin{bmatrix}
			V_t\\ \Gamma A^t
		\end{bmatrix}
		\left(I+BB^{\T}\begin{bmatrix}
			V_t\\ \Gamma A^t
			\end{bmatrix}^{\T}\left(I+\begin{bmatrix}
				T_t \\ \Gamma \Upsilon_t 
		\end{bmatrix}\begin{bmatrix}
				T_t \\ \Gamma \Upsilon_t 
	\end{bmatrix}^{\T}\right)^{-1}\begin{bmatrix}
			V_t\\ \Gamma A^t
	\end{bmatrix}  \right)^{-1}A
	\\&\clue{\cref{eq:easy}}{=}
		C^{\T}C+A^{\T}\begin{bmatrix}
			V_t\\ \Gamma A^t
			\end{bmatrix}^{\T}\left(I+\begin{bmatrix}
				T_t \\ \Gamma \Upsilon_t 
		\end{bmatrix}\begin{bmatrix}
				T_t \\ \Gamma \Upsilon_t 
	\end{bmatrix}^{\T}+\begin{bmatrix}
			V_t\\ \Gamma A^t
		\end{bmatrix}BB^{\T}\begin{bmatrix}
			V_t\\ \Gamma A^t
			\end{bmatrix}^{\T}  \right)^{-1}\begin{bmatrix}
			V_t\\ \Gamma A^t
	\end{bmatrix}A
	\\&=\begin{bmatrix}
		C\\ V_tA\\ \Gamma A^{t+1}
			\end{bmatrix}^{\T}\begin{bmatrix}
			I & \\ &
			I+\begin{bmatrix}
				T_t \\ \Gamma \Upsilon_t 
		\end{bmatrix}\begin{bmatrix}
				T_t \\ \Gamma \Upsilon_t 
	\end{bmatrix}^{\T}+\begin{bmatrix}
			V_tB\\ \Gamma A^tB
		\end{bmatrix}\begin{bmatrix}
			V_tB\\ \Gamma A^tB
			\end{bmatrix}^{\T}  \end{bmatrix}^{-1}\begin{bmatrix}
			C\\V_tA\\ \Gamma A^{t+1}
	\end{bmatrix}
	\\&=\begin{bmatrix}
		V_{t+1}\\ \Gamma A^{t+1}
		\end{bmatrix}^{\T}\left(I+ \begin{bmatrix}
			0 & 0 \\V_tB &T_t\\ \Gamma A^tB&\Gamma \Upsilon_t 
		\end{bmatrix}\begin{bmatrix}
				0 & 0 \\V_tB &T_t\\ \Gamma A^tB&\Gamma \Upsilon_t 
	\end{bmatrix}^{\T} \right)^{-1}\begin{bmatrix}
		V_{t+1}\\ \Gamma A^{t+1}
	\end{bmatrix}
	\\&=\begin{bmatrix}
		V_{t+1}\\ \Gamma A^{t+1}
		\end{bmatrix}^{\T}\left(I+ \begin{bmatrix}
			T_{t+1}\\ \Gamma  \Upsilon_{t+1} 
		\end{bmatrix}\begin{bmatrix}
			T_{t+1}\\ \Gamma  \Upsilon_{t+1} 
	\end{bmatrix}^{\T} \right)^{-1}\begin{bmatrix}
		V_{t+1}\\ \Gamma A^{t+1}
	\end{bmatrix}
	\\&=X_{t+1}
	.
	\qedhere
	\end{align*}
\end{proof}
\begin{theorem}\label{thm:form-of-solution:dare:arb}
	Let $Q_1,Q_2,Q_3,Q_4,W,\Xi_1,\Xi_2$ as in \cref{thm:form-of-solution:dare}.
	Then the sequence $X_t$ defined by \cref{eq:noniter:X:arb} can be generated by
	\begin{equation*}\label{eq:Xt=SDS:arb}
		X_t= 
			\begin{bmatrix}
				C \\ \Xi_1 \\  \Xi_2\\ \Xi_\Gamma
			\end{bmatrix}^{\T}\begin{bmatrix}
			I_{l} &&\\& (I_{t-1}\otimes Q_1)^{-1}&\\&&(I_{t-1}\otimes W)^{-1}\\&&& W_\Gamma^{-1}
			\end{bmatrix}\begin{bmatrix}
				C \\ \Xi_1 \\  \Xi_2\\ \Xi_\Gamma
			\end{bmatrix},
	\end{equation*}
	where
	\begin{align*}
		W_\Gamma &= I_{\wtd l}+\Gamma \Upsilon_t \Upsilon_t ^{\T}\Gamma ^{\T}- \Xi_{1,\Gamma}^{\T}(I\otimes Q_1)^{-1}\Xi_{1,\Gamma} -\Xi_{2,\Gamma}^{\T}(I\otimes W)^{-1}\Xi_{2,\Gamma},
		\\
		\Xi_\Gamma &= \Gamma A^t-\Xi_{1,\Gamma}^{\T}(I\otimes Q_1)^{-1}\Xi_1 -\Xi_{2,\Gamma}^{\T}(I\otimes W)^{-1}\Xi_2\in \mathbb{R}^{\wtd l\times n},
	\end{align*}
	and
	\begin{align*}
		\Xi_{1,\Gamma} &= 
		\toepU_{l\times l}\left(\begin{bmatrix}
						Q_2\\Q_1\\
				\end{bmatrix}\right)^{\T}\toepL_{l\times m}( V_{t-1}B) \Upsilon_{t-1}^{\T}A^{\T}\Gamma^{\T}\in \mathbb{R}^{(t-1)l\times \wtd l}
				,\qquad
				\\
				\Xi_{2,\Gamma} &= 
		\toepU_{l\times m}\left(\begin{bmatrix}
						Q_3\\0\\
				\end{bmatrix}\right)^{\T}\toepL_{l\times m}( V_{t-1}B) \Upsilon_{t-1}^{\T}A^{\T}\Gamma^{\T}\in \mathbb{R}^{(t-1)m\times \wtd l}
				.
	\end{align*}
\end{theorem}
\begin{proof}
	By \cref{eq:noniter:X:arb},
	\begin{align*}
		X_t
		&=\begin{bmatrix}
			V_t\\ \Gamma A^t
			\end{bmatrix}^{\T}\begin{bmatrix}
				I+T_tT_t^{\T} & T_t\Upsilon_t^{\T}\Gamma^{\T} \\ \Gamma \Upsilon_tT_t^{\T} & I+\Gamma \Upsilon_t\Upsilon_t^{\T}\Gamma^{\T}
		\end{bmatrix}^{-1}\begin{bmatrix}
			V_t\\ \Gamma A^t
		\end{bmatrix}
	\\&=
		(*)^{\T}
		\begin{bmatrix}
			I+T_t T_t ^{\T} & \\ & I+\Gamma \Upsilon_t \Upsilon_t ^{\T}\Gamma ^{\T}-\Gamma \Upsilon_t T_t ^{\T} (I+T_t T_t ^{\T})^{-1}T_t \Upsilon_t ^{\T}\Gamma ^{\T} 
		\end{bmatrix}^{-1}
		\begin{bmatrix}
		I & \\ -\Gamma \Upsilon_t T_t ^{\T}(I+T_t T_t ^{\T})^{-1} &I
		\end{bmatrix}\begin{bmatrix}
		V_t  \\ \Gamma A^t
		\end{bmatrix}
	\\&=
	\begin{bmatrix}
	V_t  \\ \Xi_\Gamma 
	\end{bmatrix}^{\T}\begin{bmatrix}
			I+T_t T_t ^{\T} & \\ & I+\Gamma \Upsilon_t \Upsilon_t ^{\T}\Gamma ^{\T}-\Gamma \Upsilon_t T_t ^{\T} (I+T_t T_t ^{\T})^{-1}T_t \Upsilon_t ^{\T}\Gamma ^{\T} 
	\end{bmatrix}^{-1}\begin{bmatrix}
	V_t  \\ \Xi_\Gamma 
	\end{bmatrix}
	\\&= V_t^{\T}(I+T_t T_t^{\T} )^{-1}V_t + \Xi_\Gamma ^{\T}\left(I+\Gamma \Upsilon_t \Upsilon_t ^{\T}\Gamma ^{\T}-\Gamma \Upsilon_t T_t ^{\T} (I+T_t T_t ^{\T})^{-1}T_t \Upsilon_t ^{\T}\Gamma ^{\T} \right)^{-1}\Xi_\Gamma 
	,
\end{align*}
	where $\Xi_\Gamma =\Gamma A^t-\Gamma \Upsilon_t T_t ^{\T}(I+T_t T_t ^{\T})^{-1}V_t $ and $*$ is used to indicate the same part limited by the symmetry.
	Similarly to \cref{eq:Xt=CTC+A},
	by \cref{eq:I+LLT:inv}, 
	writing $\toepL=\toepL_{l\times m}( V_{t-1}B)$, it can be simplified to
	\begin{align*}
		X_t
		&=
		\begin{multlined}[t]
			C^{\T}C+ A^{\T}V_{t-1}^{\T} \left(I+\toepL\toepL^{\T}\right)^{-1}V_{t-1}A
			\\ +\Xi_\Gamma ^{\T}\left(I+\Gamma \Upsilon_t \Upsilon_t ^{\T}\Gamma ^{\T}- \Gamma A\Upsilon_{t-1}\toepL ^{\T}(I+\toepL\toepL^{\T})^{-1}\toepL \Upsilon_{t-1}^{\T}A^{\T}\Gamma^{\T}\right)^{-1}\Xi_\Gamma 
		\end{multlined}
		\\&=
		\begin{multlined}[t]
			C^{\T}C+ \Xi_1^{\T}(I\otimes Q_1)^{-1}\Xi_1 +\Xi_2^{\T}(I\otimes W)^{-1}\Xi_2
		\\	+\Xi_\Gamma ^{\T}\left(I+\Gamma \Upsilon_t \Upsilon_t ^{\T}\Gamma ^{\T}- \Xi_{1,\Gamma}^{\T}(I\otimes Q_1)^{-1}\Xi_{1,\Gamma} -\Xi_{2,\Gamma}^{\T}(I\otimes W)^{-1}\Xi_{2,\Gamma}\right)^{-1}\Xi_\Gamma 
		,
\end{multlined}
	\end{align*}
	where 
	\begin{align*}
		\Xi_\Gamma 
		&=\Gamma A^t-\Gamma A\Upsilon_{t-1}\toepL ^{\T}(I+\toepL\toepL^{\T})^{-1}V_{t-1}A
	=\Gamma A^t-\Xi_{1,\Gamma}^{\T}(I\otimes Q_1)^{-1}\Xi_1 -\Xi_{2,\Gamma}^{\T}(I\otimes W)^{-1}\Xi_2
		.
		\qedhere
	\end{align*}
\end{proof}
Note that the product of two lower triangular block-Toeplitz matrices is still a lower triangular block-Toeplitz matrix.
Hence writing $\Xi_1=\begin{bmatrix}
	\Xi_{1,c}\\\Xi_{1,b}
\end{bmatrix}$ where $\Xi_{1,b}\in \mathbb{R}^{l\times n}$,
\begin{align*}
	\toepU_{l\times l}\left(\begin{bmatrix}
			Q_2\\Q_1\\
	\end{bmatrix}\right)^{\T}\toepL_{l\times m}( V_{t-1}B)
	&=\toepU_{l\times l}\left(\begin{bmatrix}
			Q_2\\Q_1\\
			\end{bmatrix}\right)^{\T}\left[I_{t-1}\otimes (CB)+\toepL_{l\times m}\left( \begin{bmatrix}
		0\\ V_{t-2}A
	\end{bmatrix}\right)I_{t-1}\otimes B\right]
	\\&=\toepU_{l\times l}\left(\begin{bmatrix}
			Q_2\\Q_1\\
			\end{bmatrix}\right)^{\T}I_{t-1}\otimes (CB)+\toepL_{l\times n}\left(\begin{bmatrix}
	0\\\Xi_{1,c}
\end{bmatrix}\right)I_{t-1}\otimes B,
\\&=\toepU_{l\times l}\left((CB)^{\T}\begin{bmatrix}
			Q_2\\Q_1\\
	\end{bmatrix}\right)^{\T}+\toepL_{l\times m}\left(\begin{bmatrix}
	0\\\Xi_{1,c}
\end{bmatrix}B\right).
\end{align*}
Similarly, writing $\Xi_2=\begin{bmatrix}
	\Xi_{2,c}\\\Xi_{2,b}
\end{bmatrix}$ where $\Xi_{2,b}\in \mathbb{R}^{l\times n}$,
\begin{align*}
	\toepU_{l\times l}\left(\begin{bmatrix}
			Q_3\\0\\
	\end{bmatrix}\right)^{\T}\toepL_{l\times m}( V_{t-1}B)
	= \toepU_{l\times l}\left((CB)^{\T}\begin{bmatrix}
				Q_3\\0\\
		\end{bmatrix}\right)^{\T}+\toepL_{l\times m}\left(\begin{bmatrix}
		0\\\Xi_{2,c}
	\end{bmatrix}B\right)
	.
\end{align*}
These can be used to reduce calculations for $\Xi_{1,\Gamma}$ and $\Xi_{2,\Gamma}$.

Which choice of the initial matrix $\Gamma$ makes the iteration \cref{eq:ddre:arb} converge?
\Cref{lm:ddre:attractor} gives an easy sufficient condition, which 
can be immediately applied to \cref{alg:fna-d}, and used to deal with the situation where $t$ must keep small, as is declared in the illustration on the parameter and output for \cref{alg:fna-d} in \cref{ssec:efficient-method}.
\begin{lemma}\label{lm:ddre:attractor}
	The unique positive semi-definite d-stabilizing solution $X_\star$ of the DARE \cref{eq:dare} is an attractor (i.e., asymptotically stable fixed point) of the DRE \cref{eq:ddre} or \cref{eq:ddre:arb}.
	Moreover, any symmetric matrix $X$ satisfying one of the two following condition lies in its attraction basin:
	\begin{enumerate}
		\item  $0\preceq X\preceq X_\star$;
		\item $\N{X-X_\star}\le \frac{1-\eta\N{A_{X_\star}}^2}{\N{B(I+B^TX_\star B)^{-1}B^T}}$ for some $\eta\in[0,1)$ and some norm $\N{\cdot}$ satisfying $\N{I}=1,\N{A_{X_\star}}<1$.
	\end{enumerate}
	As a result, \cref{eq:ddre:arb} with the matrix above as its initial term converges to $X_\star$.
\end{lemma}
\begin{proof}
	First calculate the differentials.
	\begin{align*}
		\diff\op D(X)
		&=A^{\T}\diff X(I+BB^{\T}X)^{-1}A+A^{\T}X\diff\left((I+BB^{\T}X)^{-1}\right)A \\
		&=A^{\T}\diff X(I+BB^{\T}X)^{-1}A-A^{\T}X(I+BB^{\T}X)^{-1}BB^{\T}\diff X(I+BB^{\T}X)^{-1}A
		\\&=A^{\T}\left[I-X(I+BB^{\T}X)^{-1}BB^{\T}\right]\diff X(I+BB^{\T}X)^{-1}A
		\\&=A^{\T}(I+XBB^{\T})^{-1}\diff X(I+BB^{\T}X)^{-1}A
		\\&=A_X^{\T}\diff XA_X
		,
	\end{align*}
	where $A_X$ is the closed loop matrix.
	In order to avoid the appearance of 4th-order tensor, we use vectorization to obtain
	\[
		\diff\left(\vectorize\op D(X)\right)=A_X^{\T}\otimes A_X^{\T}\diff\left(\vectorize X\right)
		\quad \text{and} \quad
		\frac{\diff\left(\vectorize\op D(X)\right)}{\diff\left(\vectorize X\right)}=A_X^{\T}\otimes A_X^{\T}
		.
	\]
	Since $X_\star$ is d-stabilizing, $\rho(A_{X_\star})<1$ and thus the Fr\'echet derivative at $X_\star$ has norm less than $1$, which guarantees $X_\star$ is an attractor.

	For the first kind of matrices, by \cref{lm:ddre:monotonic}, $X_\star=\op D^t(X_\star)\succeq \op D^t(X)\succeq \op D^t(0)\to X_\star$, which forces $\op D^t(X)\to X_\star$.
	For the second one,  
	Writing $X-X_\star=\Delta $, 
	\begin{align*}
		\N{A_{X}}
		&=\N{(I+BB^{\T}X)^{-1}A}
		\\&=\N{[I+(I+BB^{\T}X_\star)^{-1}BB^{\T}\Delta]^{-1}(I+BB^{\T}X_\star)^{-1}A}
		\\&\le\N{[I+B(I+B^{\T}X_\star B)^{-1}B^{\T}\Delta]^{-1}}\N{(I+BB^{\T}X_\star)^{-1}A}
		\\&\le\frac{1}{1-\N{B(I+B^{\T}X_\star B)^{-1}B^{\T}\Delta}}\N{(I+BB^{\T}X_\star)^{-1}A}
	\\&\le\frac{\N{A_{X_\star}}}{1-\N{B(I+B^{\T}X_\star B)^{-1}B^{\T}}\N{\Delta}}
	\\&\le\frac{\eta}{\N{A_{X_\star}}}
		.
	\end{align*}
	Then we consider
	\begin{align*}
		\op D(X)-X_\star
		&=\op D(X)-\op D(X_\star)
		\\&=A^TX(I+BB^TX)^{-1}A
		-A^T(I+X_\star BB^T)^{-1}X_\star A
		\\&=A^T(I+X_\star BB^T)^{-1}(X-X_\star)(I+BB^TX)^{-1}A
		\\&=A_{X_\star}^T (X-X_\star)A_X
		,
	\end{align*}
	which implies $\N{\op D(X)-X_\star}\le\eta \N{X-X_\star}$.
	Thus, $\op D^t(X)\to X_\star$ by reasoning in the same way consecutively.
\end{proof}

According to \cref{lm:ddre:attractor}, the compression technique can be used in \cref{alg:fna-d} without breaking its convergence. 
We roughly describe the process here:
after performing \cref{alg:fna-d} for a small/mid $t$, a truncation technique (e.g., SVD/QR) is used on $S$ to produce $S'_X=\Gamma$;
then  \cref{eq:noniter:X:arb} is used to generate a new approximation; repeat this process until convergence.
To decrease the number of calculations, the same $t$ is used in each outer iteration.

\section{CARE}\label{sec:care}
Given a linear time-invariant control system in continuous-time:
\begin{align*}
	\dot x(t) & = A x(t) + B u(t),\\
	y(t)&=C x(t),
\end{align*}
where $A\in \mathbb{R}^{n\times n}, B\in \mathbb{R}^{n\times m}, C\in \mathbb{R}^{l\times n}$.
Suppose the following condition holds through out this section:
\[
	\fbox{$(A,B)$ is c-stabilizable and $(C,A)$ is c-detectable,}
\]
or equivalently, 
$\rank(\begin{bmatrix}
	A-\lambda I & BB^{\T}
\end{bmatrix})=\rank(\begin{bmatrix}
A^{\T}-\lambda I & C^{\T}C
\end{bmatrix})=n$ for any $\lambda\in \mathbb{C}\setminus \mathbb{C}_-$,
where $\mathbb{C}_-$ is the open left half complex plane.

Its linear-quadratic optimal control can be expressed as 
\begin{equation*}\label{eq:opt-control:care}
	\arg\min_{u(t)}\int_0^\infty \left[ y(t)^{\T} y(t) + u(t)^{\T} u(t) \right] \diff t
 = -B^{\T} X_\star x(t),
\end{equation*}
where $X_\star$ is the unique symmetric positive semi-definite c-stabilizing solution $X$ of the CARE~\cite{biniIM2012numerical,chuFL2005structurepreserving,lancasterR1991solutions,mehrmann1991automomous}:
\begin{equation} \label{eq:care}
	\op C(X):=A^{\T} X + X A - X BB^{\T} X + C^{\T}C = 0.
\end{equation}
Here a solution $X$ is called c-stabilizing, if the closed loop matrix $A_{X}=A-BB^{\T}X$ is c-stable,
namely all of its eigenvalues lie in the open left half complex plane $\mathbb{C}_-$.

\subsection{FTA and its equivalence to many other methods}\label{ssec:fna-and-its-equivalence-to-many-other-methods}

Many numerical methods to solve CAREs are based on performing Cayley transformation on the associated Hamiltonian matrix
\[
	\hami  = \begin{bmatrix}
		A & BB^{\T} \\ C^{\T}C & -A^{\T}
	\end{bmatrix}
	.
\]
For example,  the SDA \cref{eq:sda} for DAREs is also valid for CAREs.
If the initial terms are set by
\begin{equation*}\label{eq:initial:care}
	A_0=I+2\gamma K_\gamma^{-\T}, \qquad 
		G_0=2\gamma \what A_\gamma^{-1}BB^{\T}K_\gamma^{-1},\qquad 	
		H_0=2\gamma K_\gamma^{-1}C^{\T}C\what A_\gamma^{-1},
\end{equation*}
where $\what A_\gamma=A-\gamma I, \gamma>0$ and $K_\gamma=\what A_\gamma^{\T}+C^{\T}C\what A_\gamma^{-1}BB^{\T}$,
then $I_n + G_k H_k$ are nonsingular for $k\ge 0$, and it holds that
$A_k \to  0$, $G_k \to  Y_\star$ (the solution to the dual CARE) and $H_k \to  X_\star$, all quadratically \cite{linX2006convergence}. 
The special forms of terms $A_0, G_0, H_0$ are given by the Cayley transformation $\hami \mapsto (\hami +\gamma I)(\hami -\gamma I)^{-1}$ in order to generate a structured symplectic matrix pair, see, e.g., in \cite[Section~5.3]{huangLL2018structurepreserving}.

Since the recursion of SDAs for DAREs and CAREs are the same, comparing the initial terms, it is clear that the SDA for CAREs is calculating a subsequence generated by
this DRE:
\begin{equation}\label{eq:fixedpoint:care}
	X_0=0,\qquad X_{t+1}= H_0 + A_0^{\T}X_t\left(I+G_0X_t\right)^{-1}A_0.
\end{equation}
Write $Y_\gamma=C\what A_\gamma^{-1}B$, $B_\gamma=\sqrt{2\gamma}\what A_\gamma^{-1}B(I+Y_\gamma^{\T}Y_\gamma)^{-1/2}$, $C_\gamma=\sqrt{2\gamma}(I+Y_\gamma Y_\gamma^{\T})^{-1/2}C\what A_\gamma^{-1}$, $A_\gamma=I+2\gamma \what A_\gamma^{-1}-B_\gamma Y_\gamma^{\T}C_\gamma$,
and then
\begin{align*}
	G_0
	&=2\gamma \what A_\gamma^{-1}BB^{\T}(\what A_\gamma^{\T}+C^{\T}C\what A_\gamma^{-1}BB^{\T})^{-1}
	\\&=2\gamma \what A_\gamma^{-1}B(I +B^{\T}\what A_\gamma^{-\T}C^{\T}C\what A_\gamma^{-1}B)^{-1}B^{\T}\what A_\gamma^{-\T}
		=B_\gamma B_\gamma^{\T}
		,
\end{align*}
	and similarly
	$H_0=C_\gamma^{\T}C_\gamma,
	A_0=A_\gamma$.
	%
The DRE \cref{eq:fixedpoint:care} reads
\begin{equation}\label{eq:ddre:care}
	X_0=0,\qquad X_{t+1}= C_{\gamma}^{\T}C_{\gamma} + A_{\gamma}^{\T}X_t\left(I+B_{\gamma}B_{\gamma}^{\T}X_t\right)^{-1}A_{\gamma}.
\end{equation}
The form of \cref{eq:ddre:care} coincides with \cref{eq:ddre}. Hence the discussions on DAREs can be adopted to CAREs.

Wong and Balakrishnan \cite{wongB2005quadratic,wongB2007fast} proposed the quadratic ADI method
\begin{align*}
	X_0&=0,\\
	X_{t+1/2}(A-\ol \gamma_{t+1}I-BB^{\T}X_t)&=-C^{\T}C-(A^{\T}+\ol \gamma_{t+1}I)X_t,\qquad \Re\gamma_{t+1}>0,\\
	(A^{\T}-\gamma_{t+1}I-X_{t+1/2}BB^{\T})X_{t+1}&=-C^{\T}C-X_{t+1/2}(A+\gamma_{t+1} I),
\end{align*}
and presented in \cite[(10)]{wongB2005quadratic}
\[
	X_{t+1}=C_{\gamma_{t+1}}^{\T}C_{\gamma_{t+1}}+A_{\gamma_{t+1}}^{\T}X_t\left(I+B_{\gamma_{t+1}}B_{\gamma_{t+1}}^{\T}X_t\right)^{-1}A_{\gamma_{t+1}},
\]
which is the same as \cref{eq:ddre:care} as long as $\gamma_{t+1}=\gamma$.

Lin and Simoncini \cite{linS2015new} developed the Cayley transformed Hamiltonian subspace iteration
\begin{align*}
	\begin{bmatrix}
		M_{t+1}\\ N_{t+1}
	\end{bmatrix}
	&=(\hami -\gamma_{t+1} I)^{-1}(\hami +\ol\gamma_{t+1} I)\begin{bmatrix}
		I\\-X_t
	\end{bmatrix},
	\qquad \Re\gamma_{t+1}<0,
	\\
	X_{t+1}
	&=-N_{t+1}M_{t+1}^{-1}.
\end{align*}

Benner et al.\ \cite{bennerBKS2018radi} devoted the RADI method, originated from the incorporation technique (which we will illustrate later),
and proved that 
if the initial approximation is $0$ and the same shifts are used,
the RADI method is equivalent to quadratic ADI method and Cayley transformed Hamiltonian subspace iteration, together with invariant subspace approach in \cite{amodeiB2010invariant,bennerB2016solution} for special cases.

From the analysis above, we can conclude that the FTA in this paper and the SDA are also equivalent to these methods under the same condition that the initial approximation is $0$ and the shift $\gamma$ is consistently used, in the sense that they all produce the same sequence (subsequence for SDA).

Clearly the FTA, \cref{alg:fna-d}, can be performed on the corresponding DARE to obtain the solution of the CARE.


Rather than directly using the results in \cref{ssec:in-the-view-of-matrix-analysis}, we borrow the same analysis there and eventually obtain the following analogies of 
\cref{thm:fixedpoint,thm:form-of-solution:dare}.

\begin{theorem}\label{thm:fixedpoint:care}
	Let $\wtd A=I+2\gamma \what A_\gamma^{-1},\wtd B=\sqrt{2\gamma}\what A_\gamma^{-1}B,\wtd C=\sqrt{2\gamma}C\what A_\gamma^{-1}$.
	Write
	\begin{equation*}\label{eq:noniter:X:UVT:care}
		\wtd V_t = \begin{bmatrix}
			\wtd C \\ \wtd C\wtd A \\ \wtd C\wtd A^2\\ \vdots \\ \vdots \\ \wtd C\wtd A^{t-1}
		\end{bmatrix} 
		,\qquad
		\wtd T_t = \begin{bmatrix}
			Y_\gamma          &        &         &        &    &   \\
			\wtd C\wtd B        & Y_\gamma       &         &        &    &   \\
			\wtd C\wtd A\wtd B       & \wtd C\wtd B     & \ddots  &        &    &   \\
			\vdots    &        & \ddots  & \ddots &    &   \\
			\vdots    &        &         & \wtd C\wtd B     & Y_\gamma   &   \\
			\wtd C\wtd A^{t-2}\wtd B & \cdots & \cdots  & \wtd C\wtd A\wtd B    & \wtd C\wtd B & Y_\gamma  \\
		\end{bmatrix}
		,
	\end{equation*}
	$\wtd T_1=Y_\gamma $. 
	Then the terms of the sequence $\set{X_t}$ generated by the DRE \cref{eq:ddre:care} are
	\begin{equation}\label{eq:noniter:X:X:care}
		X_t
		= \wtd V_t^{\T}(I+\wtd T_t\wtd T_t^{\T})^{-1}\wtd V_t
		, \qquad t=1,2,\dots.
	\end{equation}
	Moreover, $\set{X_t}$ is monotonically nondecreasing, and $X_t\to X_{\star}$, the solution of CARE \cref{eq:care}.
\end{theorem}
\begin{proof}
	The monotonicity and the convergence of the sequence are the same as that for DAREs in \cref{ssec:in-the-view-of-matrix-analysis} and hence omitted. Only \cref{eq:noniter:X:X:care} is proved here.

	It is easy to verify that \cref{eq:noniter:X:X:care} is correct for $t=1$.
	Assuming \cref{eq:noniter:X:X:care} is correct for $t$, we are going to prove it is also correct for $t+1$.
	By the DRE \cref{eq:ddre:care},
\begin{align*}
	X_{t+1}
	&=
	\begin{multlined}[t]
		\wtd C^{\T}(I+Y_\gamma Y_\gamma ^{\T})^{-1}\wtd C + (\wtd A-\wtd BY_\gamma ^{\T}(I+Y_\gamma Y_\gamma ^{\T})^{-1}\wtd C)^{\T}\wtd V_t^{\T}\left(I+\wtd T_t\wtd T_t^{\T}\right)^{-1}\wtd V_t
		\\\cdot\left(I+\wtd B(I+Y_\gamma ^{\T}Y_\gamma )^{-1}\wtd B^{\T}\wtd V_t^{\T}\left(I+\wtd T_t\wtd T_t^{\T}\right)^{-1}\wtd V_t\right)^{-1}(\wtd A-\wtd BY_\gamma ^{\T}(I+Y_\gamma Y_\gamma ^{\T})^{-1}\wtd C)
	\end{multlined}
	\\&\clue{\cref{eq:easy}}{=}
	\begin{multlined}[t]
		\wtd C^{\T}(I+Y_\gamma Y_\gamma ^{\T})^{-1}\wtd C + (\wtd A-\wtd BY_\gamma ^{\T}(I+Y_\gamma Y_\gamma ^{\T})^{-1}\wtd C)^{\T}\wtd V_t^{\T}
		\\\cdot\left(I+\wtd T_t\wtd T_t^{\T}+\wtd V_t\wtd B(I+Y_\gamma ^{\T}Y_\gamma )^{-1}\wtd B^{\T}\wtd V_t^{\T}\right)^{-1}\wtd V_t(\wtd A-\wtd BY_\gamma ^{\T}(I+Y_\gamma Y_\gamma ^{\T})^{-1}\wtd C)
	\end{multlined}
		\\&=
		(*)^{\T}\begin{bmatrix}
			I+Y_\gamma Y_\gamma ^{\T} & \\  & I+\wtd T_t\wtd T_t^{\T}+\wtd V_t\wtd B(I+Y_\gamma ^{\T}Y_\gamma )^{-1}\wtd B^{\T}\wtd V_t^{\T}
			\end{bmatrix}^{-1}
	\begin{bmatrix}
			\wtd C\\ \wtd V_t\wtd A-\wtd V_t\wtd BY_\gamma ^{\T}(I+Y_\gamma Y_\gamma ^{\T})^{-1}\wtd C
		\end{bmatrix}
		\\&= 
		\begin{multlined}[t]
		(*)^{\T}\begin{bmatrix}
			I+Y_\gamma Y_\gamma ^{\T} & \\  & I+\wtd T_t\wtd T_t^{\T}+\wtd V_t\wtd B\wtd B^{\T}\wtd V_t^{\T}-\wtd V_t\wtd BY_\gamma ^{\T}(I+Y_\gamma Y_\gamma ^{\T})^{-1}Y_\gamma \wtd B^{\T}\wtd V_t^{\T}
			\end{bmatrix}^{-1}
			\\\cdot\begin{bmatrix}
			I & \\ -\wtd V_t\wtd BY_\gamma ^{\T}(I+Y_\gamma Y_\gamma ^{\T})^{-1} & I
			\end{bmatrix}\begin{bmatrix}
			\wtd C\\ \wtd V_t\wtd A
		\end{bmatrix}
		\end{multlined}
		\\&= \begin{bmatrix}
			\wtd C\\ \wtd V_t\wtd A
			\end{bmatrix}^{\T}\begin{bmatrix}
				I+Y_\gamma Y_\gamma ^{\T} & Y_\gamma \wtd B^{\T}\wtd V_t^{\T}\\ \wtd V_t\wtd BY_\gamma ^{\T} & I+\wtd T_t\wtd T_t^{\T}+\wtd V_t\wtd B\wtd B^{\T}\wtd V_t^{\T}
			\end{bmatrix}^{-1}\begin{bmatrix}
			\wtd C\\ \wtd V_t\wtd A
		\end{bmatrix}
	\\&=\begin{bmatrix}
			\wtd C\\ \wtd V_t\wtd A
			\end{bmatrix}^{\T}\left(I+\begin{bmatrix}
				Y_\gamma  & 0\\ \wtd V_t\wtd B & \wtd T_t
				\end{bmatrix}\begin{bmatrix}
				Y_\gamma  & 0\\ \wtd V_t\wtd B & \wtd T_t
			\end{bmatrix}^{\T}\right)^{-1}\begin{bmatrix}
			\wtd C\\ \wtd V_t\wtd A
		\end{bmatrix}
	\\&=\wtd V_{t+1}^{\T}(I+\wtd T_{t+1}\wtd T_{t+1}^{\T})^{-1}\wtd V_{t+1}
	.
\end{align*}
Here $*$ is still used to indicate the same part limited by the symmetry.
\end{proof}

Using the notations for Toeplitz matrices in \cref{sec:preliminary},
 we have 
\[
	\wtd T_t
	=\toepL_{l\times m}\left(\begin{bmatrix}
		Y_\gamma \\ \wtd V_{t-1}\wtd B\\
	\end{bmatrix}\right)
	=I_t\otimes Y_\gamma +\begin{bmatrix}
		0 & 0\\
		\toepL_{l\times m}( \wtd V_{t-1}\wtd B) & 0 \\
	\end{bmatrix}=\begin{bmatrix}
		Y_\gamma  & 0\\ \wtd V_{t-1}\wtd B & \wtd T_{t-1}
	\end{bmatrix}
	.
\]
\begin{theorem}\label{thm:form-of-solution:care}
Let 
\begin{subequations}\label{eq:Q:care}
	\begin{align}
		\left(I_{tl}+\wtd T_t\wtd T_t^{\T}\right)\begin{bmatrix}
					Q_2 \\Q_1\\
				\end{bmatrix}&=
			\begin{bmatrix}
				0 \\ I_l \\
			\end{bmatrix}
			,\qquad Q_1\in \mathbb{R}^{l\times l}
			,\\
			\left(I_{(t-1)l}+\wtd V_{t-1}\wtd B\wtd B^{\T}\wtd V_{t-1}^{\T}+\wtd T_{t-1}\wtd T_{t-1}^{\T}\right)Q_3&=
				 \wtd V_{t-1}\wtd B
			,
	\end{align}
\end{subequations}
and $W=I_{m}-Q_3^{\T} \wtd V_{t-1}\wtd B$.
	Then the sequence $X_t$ defined by DRE \cref{eq:noniter:X:X:care} can be generated by
	\begin{equation}\label{eq:Xt=SDS:care}
		X_t= 
			\begin{bmatrix}
				 \Xi_1 \\  \Xi_2
			\end{bmatrix}^{\T}\begin{bmatrix}
			(I_{t}\otimes Q_1)^{-1}&\\&(I_{t}\otimes \left[W+WY_\gamma ^{\T}Y_\gamma W\right])^{-1}
			\end{bmatrix}\begin{bmatrix}
				 \Xi_1 \\  \Xi_2
			\end{bmatrix},
	\end{equation}
	where
	\[
		\Xi_1 = \toepU_{l\times l}\left(\begin{bmatrix}
				Q_2\\Q_1\\
		\end{bmatrix}\right)^{\T}\wtd V_t\in \mathbb{R}^{tl\times n},
		\qquad
		\Xi_2=\toepU_{l\times m}\left(\begin{bmatrix}
				Q_3\\0\\
		\end{bmatrix}\right)^{\T} \wtd V_t\in \mathbb{R}^{tm\times n}
		.
	\]
\end{theorem}
\begin{proof}
	By \cref{lm:form-of-inverse:dare} with $Y\leftarrow Y_\gamma , D_{t-1}\leftarrow \wtd V_{t-1}\wtd B$,
	\begin{equation}\label{eq:I+TT:inv:care}
		\begin{multlined}[t]
			(I+\wtd T_t\wtd T_t^{\T})^{-1}
			=\toepU_{l\times m}\left(\begin{bmatrix}
					Q_2\\Q_1\\
					\end{bmatrix}\right)(I\otimes Q_1)^{-1}\toepU_{l\times m}\left(\begin{bmatrix}
					Q_2\\Q_1\\
			\end{bmatrix}\right)^{\T}
			\\\hspace{2cm}+\toepU_{l\times m}\left(\begin{bmatrix}
					Q_3\\0\\
					\end{bmatrix}\right)\left(I\otimes \left[W+WY_\gamma ^{\T}Y_\gamma W\right]\right)^{-1}\toepU_{l\times m}\left(\begin{bmatrix}
					Q_3\\0\\
			\end{bmatrix}\right)^{\T}
			.
			\end{multlined}
	\end{equation}
Then the result follows from \cref{eq:noniter:X:X:care}.
\end{proof}
\Cref{thm:form-of-solution:care} suggests a similar algorithm, \cref{alg:fna-c}, to approximate the solution of CAREs.

\begin{algorithm}[h]
	\caption{FFT-based Toeplitz-structured Approximation (FTA) for CAREs}\label{alg:fna-c}
	\begin{algorithmic}[1]
		\REQUIRE $A \in \mathbb{R}^{n\times n},B \in \mathbb{R}^{n\times m},C \in \mathbb{R}^{l\times n}$ and $\gamma,t$.
		\STATE Compute $\what A_\gamma = A-\gamma I$ and generate the linear-system solver $\what A_\gamma^{-1}$ for the sparse $A$ (or its PLU factorization $\what A_\gamma = PLU$ for the dense $A$).
		\STATE Compute $\mathtt{(tmp)}=U^{-1}L^{-1}P^{-1}B,\wtd C=\sqrt{2\gamma}C\what A_\gamma^{-1}$
		by the linear solver or forward/backward substitution, and compute $\wtd B = \sqrt{2\gamma}\mathtt{(tmp)}, Y_\gamma =C\mathtt{(tmp)}$.
		\STATE Compute sequentially 
		$\wtd C\cdot \wtd A, \wtd C\wtd A\cdot \wtd A,\dots,\wtd C\wtd A^{t-3} \cdot \wtd A, \wtd C\wtd A^{t-2}\cdot \wtd A$
		by the way $\mathtt{(tmp)}\wtd A=\mathtt{(tmp)}+2\gamma \mathtt{(tmp)}\what A_\gamma^{-1}$ 
		by the linear solver or forward/backward substitution, and form $\wtd V_t\in \mathbb{R}^{tl\times n}$ by stacking $\wtd C$ and the $t-1$ terms vertically in order, where the first $t-1$ terms consists of $\wtd V_{t-1}$.
		\STATE Compute $\wtd V_{t-1}\wtd B\in \mathbb{R}^{(t-1)l\times m}$ and form $\begin{bmatrix}
			Y_\gamma \\\wtd V_{t-1}\wtd B
		\end{bmatrix}\in \mathbb{R}^{tl\times m}$.
		\STATE Use Preconditioned Conjugate Gradient method to solve \cref{eq:Q:care}.
		\STATE Compute $W=I_{m}-Q_3^{\T}\wtd V_{t-1}\wtd B, \mathtt{(tmp)}=Y_\gamma W, \wtd W= W+\mathtt{(tmp)}^{\T}\mathtt{(tmp)}$ and then the Cholesky factorizations of $Q_1=L_QL_Q^{\T}$ and $\wtd W=L_WL_W^{\T}$.
		\STATE Use fast multiplication to obtain $S_1=\toepU_{l\times l}\left(\begin{bmatrix}
				Q_2\\Q_1\\
				\end{bmatrix}L_Q^{-\T}\right)^{\T}\wtd V_t\in \mathbb{R}^{tl\times n}$ and $S_2=\toepU_{l\times m}\left(\begin{bmatrix}
				Q_3\\0\\
		\end{bmatrix}L_W^{-\T}\right)^{\T} \wtd V_t\in \mathbb{R}^{tm\times n}$,
		and form $S=\begin{bmatrix}
			 S_1\\ S_2
		 \end{bmatrix}\in \mathbb{R}^{t(l+m)\times n}$.
		\ENSURE $S$ which satisfies $S^{\T}S\approx X_\star\in \mathbb{R}^{n\times n}$.
	\end{algorithmic}
\end{algorithm}

Some remarks are given below to illustrate the algorithm.
\paragraph{Parameter and output}
\begin{enumerate}
	\item Considerations similar to \cref{alg:fna-d} have to be made.
\Cref{eq:Xt=SDS:care} is not a compact form either,
and some truncation/reduction/shrinking technique may be brought in.
\end{enumerate}

\paragraph{Time complexity}
Complexity for $S$, the factor of $X_t$:
\begin{enumerate}
	\item Step~1, compute $\what A_\gamma$ and its PLU factorization in $n+\frac{n(n-1)(4n+1)}{6}$ flops.
	\item Step~2, compute $\wtd C, \wtd B, Y_\gamma $ in $2n^2m+2n^2l+nm+nl+lm(2n-1)$ flops.
	\item Step~3, compute $\wtd V_t$, namely $\wtd C\wtd A,\wtd C\wtd A^2,\dots,\wtd C\wtd A^{t-1}$, in $(t-1)[2n^2l+nl+nl]$ flops. 
	\item Step~4, compute $\wtd V_{t-1}\wtd B$ in $(t-1)(2n-1)lm$ flops.
	\item Step~5, Use $M$-step PCG (suppose one-step PCG is done in $Ntl\ln l$ flops for fast multiplication where $N$ is a constant), to compute $Q_{*,*}$, in $\OO(MNtl[\ln^2t+\ln(tl)](l+m))$ flops.
	\item Step~6, compute $W,\wtd W$ in $(2tl-1)\frac{m(m+1)}{2}+m+(2m-1)lm+(2l-1)\frac{m(m+1)}{2}$ flops, and $L_Q,L_W$ in $\frac{2}{3}l(l-1)(l+4)$ flops.
	\item Step~7, compute $\begin{bmatrix}
				Q_2\\Q_1\\
				\end{bmatrix}L_Q^{-\T},\begin{bmatrix}
				Q_3\\0\\
			\end{bmatrix}L_W^{-\T}$ in $tl^3+(t-1)lm^2$ flops; compute $S_1,S_2$ in $\OO(2Nntl\ln l)$ flops.
	\item To sum up, assuming $l\ll n, m\ll n$ and omitting lower order terms, the total complexity is $\frac{2}{3}n^3+2(l+m)n^2+2lmn+2ln^2t+2lmnt+\OO(MNl(l+m)t\ln^2t+2Nltn\ln l)=\frac{2}{3}n^3+\OO\left(t(n^2+\ln^2t)\right)$ flops.
	\item Suppose $A$ is sparse, and the number of nonzero entries is $\nnz(A)$.
		The PLU factorization in Step~1 can be replaced by an iterative solver with at most $P$ steps, such as CG, MINRES and GMRES.
		The total complexity is $\OO(P(m+l)\nnz(A)+lmn+lt\nnz(A)+lmnt+MNl(l+m)t\ln^2t+2Nltn\ln l)=\OO\left(t(\nnz(A)+n+\ln^2t)\right)$ flops.
		It is worthwhile to mention that computing $\what A_\gamma^{-1}$ or solving the corresponding linear systems is necessary for all methods like RADI and the Cayley transformed Hamiltonian subspace iteration.
\end{enumerate}
\paragraph{Space complexity}
\begin{enumerate}
	\item The storage is similar to that of \cref{alg:fna-d}.
\end{enumerate}

\subsection{Incorporation technique}\label{ssec:incorporation-technique}

In the following, we consider the incorporation technique (a.k.a. defect correction).
The key idea is: once an approximate solution $\wtd X$ is obtained, letting the difference from the exact solution $X_{\star}$ be $\Delta$, namely $X_{\star}=\wtd X+\Delta$, the difference satisfies $A^{\T}(\wtd X+\Delta)+(\wtd X+\Delta)A+C^{\T}C-(\wtd X+\Delta)BB^{\T}(\wtd X+\Delta)=0$, from which an approximation $\wtd\Delta$ can be generated and then $\wtd X+\wtd \Delta$ should be an approximate solution to the original equation better than $\wtd X$.
More details can be found in \cite{huangLL2018structurepreserving,bennerBKS2018radi}.
The following lemma is important as the guarantee of the validity of the incorporation technique.
\begin{lemma}[{\cite[Theorem~1]{bennerBKS2018radi}}]\label{lm:incorp}
	Let $\wtd X$ be an approximation to a solution to \cref{eq:care}.
	\begin{enumerate}
		\item $\Delta=X_{\star}-\wtd X$ is a solution to the equation
			\begin{equation}\label{eq:incorp}
				(A-BB^{\T}\wtd X)^{\T}\Delta+\Delta(A-BB^{\T}\wtd X)+\op C(\wtd X)-\Delta BB^{\T}\Delta=0.
			\end{equation}
		\item Conversely, if $\Delta$ is a solution to \cref{eq:incorp}, then $\wtd X+\Delta$ is a solution to \cref{eq:care}.
			Moreover, if $\wtd X\succeq0$ and $\Delta$ is a c-stabilizing solution to \cref{eq:incorp}, then $\wtd X+\Delta$ is the c-stabilizing solution to \cref{eq:care}.
		\item If $\wtd X\succeq0,\op C(\wtd X)\succeq0$, then $\Delta$ is the unique c-stabilizing solution to \cref{eq:incorp}.
		\item If $\wtd X\succeq0,\op C(\wtd X)\succeq0$, then $\Delta\preceq X_{\star}$.
	\end{enumerate}
\end{lemma}
To make the incorporation technique useful for \cref{alg:fna-c}, 
the fundamental problem we face is the low rank factorization of $\op C(\wtd X)$.
Let $\wtd X=X_t$ that we have obtained.
\begin{theorem}\label{thm:incorp:care}
Let $\bs 1_t\in \mathbb{R}^{t}$ be a vector with each entry one.
Then for $X_t$ defined by DRE \cref{eq:ddre:care}, $\op C(X_t)=C_t^{\T}C_t$, where 
\begin{equation}\label{eq:thm:incorp:care}
	C_0=C,\qquad C_t=C+\sqrt{2\gamma}(\bs 1_t^{\T}\otimes I_l)(I+\wtd T_t\wtd T_t^{\T})^{-1}\wtd V_t.
\end{equation}
\end{theorem}
\begin{proof}
Clearly,
\begin{align*}
	\op C(X_t)
	&=A^{\T}X_t+X_tA+C^{\T}C-X_tBB^{\T}X_t
	\\&=
	\begin{multlined}[t]
	[\gamma I+2\gamma(\wtd A-I)^{-\T}]X_t+X_t[\gamma I+2\gamma(\wtd A-I)^{-1}]
	+2\gamma (\wtd A-I)^{-\T}\wtd C^{\T}\wtd C(\wtd A-I)^{-1}
	\\-2\gamma X_t(\wtd A-I)^{-1}\wtd B\wtd B^{\T}(\wtd A-I)^{-\T}X_t
	\end{multlined}
	\\&= 
	2\gamma \left[
	(\wtd A-I)^{-\T}\wtd C^{\T}\wtd C(\wtd A-I)^{-1}+(\wtd A-I)^{-\T}X_t+X_t(\wtd A-I)^{-1}+X_t 
	-X_t(\wtd A-I)^{-1}\wtd B\wtd B^{\T}(\wtd A-I)^{-\T}X_t
	\right]
	.
\end{align*}
Note that by \cref{eq:noniter:X:X:care}
\begin{align*}
	X_t(\wtd A-I)^{-1}\wtd B
	&=\wtd V_t^{\T}(I+\wtd T_t\wtd T_t^{\T})^{-1}\wtd V_t(\wtd A-I)^{-1}\wtd B
	\\&=\wtd V_t^{\T}(I+\wtd T_t\wtd T_t^{\T})^{-1}\begin{bmatrix}
			\wtd C \\ \wtd C\wtd A \\ 
			\vdots  \\ \wtd C\wtd A^{t-1}
	\end{bmatrix}(\wtd A-I)^{-1}\wtd B
	\\&=\wtd V_t^{\T}(I+\wtd T_t\wtd T_t^{\T})^{-1}\begin{bmatrix}
			Y_\gamma  \\ \wtd C\wtd B+Y_\gamma  \\ 
			\vdots  \\ \wtd C\wtd A^{t-2}\wtd B+\dots+\wtd C\wtd B+Y_\gamma 
	\end{bmatrix}
	\\&=\wtd V_t^{\T}(I+\wtd T_t\wtd T_t^{\T})^{-1}\wtd T_t\begin{bmatrix}
		I _m\\ I _m\\ \vdots \\I_m
	\end{bmatrix}
	=\wtd V_t^{\T}(I+\wtd T_t\wtd T_t^{\T})^{-1}\wtd T_t(\bs 1_t\otimes I_m)
	,
\end{align*}
and
\begin{align*}
	X_t(\wtd A-I)^{-1}
	&=\wtd V_t^{\T}(I+\wtd T_t\wtd T_t^{\T})^{-1}\begin{bmatrix}
			\wtd C \\ \wtd C\wtd A \\ \vdots \\ \wtd C\wtd A^{t-1}
	\end{bmatrix}(\wtd A-I)^{-1}
	\\&=\wtd V_t^{\T}(I+\wtd T_t\wtd T_t^{\T})^{-1}\begin{bmatrix}
			\wtd C(\wtd A-I)^{-1} \\ \wtd C+\wtd C(\wtd A-I)^{-1} \\ \vdots \\ \wtd C\wtd A^{t-2}+\dots+\wtd C+\wtd C(\wtd A-I)^{-1}
	\end{bmatrix}
	\\&=\wtd V_t^{\T}(I+\wtd T_t\wtd T_t^{\T})^{-1}\left(  (\bs 1_t\otimes I_l)\wtd C(\wtd A-I)^{-1}+\toepL_{l\times n}\left(\begin{bmatrix}
			0\\\wtd V_{t-1}
\end{bmatrix}\right)(\bs 1_t\otimes I_l)\right)
	\\&=\wtd V_t^{\T}(I+\wtd T_t\wtd T_t^{\T})^{-1}\left(  (\bs 1_t\otimes I_l)\wtd C(\wtd A-I)^{-1}+\toepL_{l\times l}\left(\begin{bmatrix}
		0\\\bs 1_{t-1}\otimes I_l
\end{bmatrix}\right)\wtd V_t\right)
,
\end{align*}
of which the last equality is guaranteed by
\[
\toepL_{l\times n}\left(\begin{bmatrix}
			0\\\wtd V_{t-1}
\end{bmatrix}\right)(\bs 1_t\otimes I_l)
=\begin{bmatrix}
	0\\ \wtd C\\ \wtd C\wtd A+\wtd C\\\vdots\\\wtd C\wtd A^{t-2}+\dots+\wtd C
\end{bmatrix}=
\toepL_{l\times l}\left(\begin{bmatrix}
		0\\\bs 1_{t-1}\otimes I_l
\end{bmatrix}\right)\wtd V_t
.
\]
Hence
\begin{align*}
	\frac{1}{2\gamma}\left(\op C(X_t)-C_t^{\T}C_t\right)
	&=
	\begin{multlined}[t]
		\ul{(\wtd A-I)^{-\T}\wtd C^{\T}\wtd C(\wtd A-I)^{-1}}
		\\+\ul{(\wtd A-I)^{-\T}\wtd C^{\T}(\bs 1_t^{\T}\otimes I_l)(I+\wtd T_t\wtd T_t^{\T})^{-1}\wtd V_t}
+\wtd V_t^{\T}\toepL_{l\times l}\left(\begin{bmatrix}
		0\\\bs 1_{t-1}\otimes I_l
\end{bmatrix}\right)^{\T}(I+\wtd T_t\wtd T_t^{\T})^{-1}\wtd V_t
\\+\ul{\wtd V_t^{\T}(I+\wtd T_t\wtd T_t^{\T})^{-1}(\bs 1_t\otimes I_l)\wtd C(\wtd A-I)^{-1}}
+\wtd V_t^{\T}(I+\wtd T_t\wtd T_t^{\T})^{-1}\toepL_{l\times l}\left(\begin{bmatrix}
		0\\\bs 1_{t-1}\otimes I_l
\end{bmatrix}\right)\wtd V_t
\\+\wtd V_t^{\T}(I+\wtd T_t\wtd T_t^{\T})^{-1}\wtd V_t  -\wtd V_t^{\T}(I+\wtd T_t\wtd T_t^{\T})^{-1}\wtd T_t(\bs 1_t\otimes I_m)(\bs 1_t^{\T}\otimes I_m)\wtd T_t^{\T}(I+\wtd T_t\wtd T_t^{\T})^{-1}\wtd V_t
\\-\ul{(\wtd A-I)^{-\T}\wtd C^{\T}\wtd C(\wtd A-I)^{-1}}
	-\wtd V_t^{\T}(I+\wtd T_t\wtd T_t^{\T})^{-1}(\bs 1_t\otimes I_l)(\bs 1_t^{\T}\otimes I_l)(I+\wtd T_t\wtd T_t^{\T})^{-1}\wtd V_t
	\\-\ul{(\wtd A-I)^{-\T}\wtd C^{\T}(\bs 1_t^{\T}\otimes I_l)(I+\wtd T_t\wtd T_t^{\T})^{-1}\wtd V_t}
	-\ul{\wtd V_t^{\T}(I+\wtd T_t\wtd T_t^{\T})^{-1}(\bs 1_t\otimes I_l)\wtd C(\wtd A-I)^{-1}}
	\end{multlined}
	\\&=
	\begin{multlined}[t]
\wtd V_t^{\T}\toepL_{l\times l}\left(\begin{bmatrix}
		0\\\bs 1_{t-1}\otimes I_l
\end{bmatrix}\right)^{\T}(I+\wtd T_t\wtd T_t^{\T})^{-1}\wtd V_t
+\wtd V_t^{\T}(I+\wtd T_t\wtd T_t^{\T})^{-1}\toepL_{l\times l}\left(\begin{bmatrix}
		0\\\bs 1_{t-1}\otimes I_l
\end{bmatrix}\right)\wtd V_t
\\+\wtd V_t^{\T}(I+\wtd T_t\wtd T_t^{\T})^{-1}\wtd V_t  -\wtd V_t^{\T}(I+\wtd T_t\wtd T_t^{\T})^{-1}\wtd T_t(\bs 1_t\otimes I_m)(\bs 1_t^{\T}\otimes I_m)\wtd T_t^{\T}(I+\wtd T_t\wtd T_t^{\T})^{-1}\wtd V_t
\\
	-\wtd V_t^{\T}(I+\wtd T_t\wtd T_t^{\T})^{-1}(\bs 1_t\otimes I_l)(\bs 1_t^{\T}\otimes I_l)(I+\wtd T_t\wtd T_t^{\T})^{-1}\wtd V_t
	\end{multlined}
	\\&=
	\begin{multlined}[t]
		\wtd V_t^{\T}(I+\wtd T_t\wtd T_t^{\T})^{-1}\Bigg[
			(I+\wtd T_t\wtd T_t^{\T})\toepL_{l\times l}\left(\begin{bmatrix}
		0\\\bs 1_{t-1}\otimes I_l
\end{bmatrix}\right)^{\T}
+\toepL_{l\times l}\left(\begin{bmatrix}
		0\\\bs 1_{t-1}\otimes I_l
\end{bmatrix}\right)(I+\wtd T_t\wtd T_t^{\T})
\\+(I+\wtd T_t\wtd T_t^{\T})  -\wtd T_t(\bs 1_t\otimes I_m)(\bs 1_t^{\T}\otimes I_m)\wtd T_t^{\T}
	-(\bs 1_t\otimes I_l)(\bs 1_t^{\T}\otimes I_l)
		\Bigg](I+\wtd T_t\wtd T_t^{\T})^{-1}\wtd V_t
	\end{multlined}
	\\&=
	\begin{multlined}[t]
		\wtd V_t^{\T}(I+\wtd T_t\wtd T_t^{\T})^{-1}\Bigg[
			\wtd T_t\wtd T_t^{\T}\toepL_{l\times l}\left(\begin{bmatrix}
		0\\\bs 1_{t-1}\otimes I_l
\end{bmatrix}\right)^{\T}
+\toepL_{l\times l}\left(\begin{bmatrix}
		0\\\bs 1_{t-1}\otimes I_l
\end{bmatrix}\right)\wtd T_t\wtd T_t^{\T}
\\+\wtd T_t\wtd T_t^{\T}  -\wtd T_t(\bs 1_t\otimes I_m)(\bs 1_t^{\T}\otimes I_m)\wtd T_t^{\T}
		\Bigg](I+\wtd T_t\wtd T_t^{\T})^{-1}\wtd V_t
	\end{multlined}
	\\&=0
	,
\end{align*}
of which the last equality holds for
\begin{equation*}\label{eq:toepL:commute}
	\toepL_{l\times l}\left(\begin{bmatrix}
			0\\\bs 1_{t-1}\otimes I_l
	\end{bmatrix}\right)\wtd T_t=\wtd T_t\toepL_{m\times m}\left(\begin{bmatrix}
			0\\\bs 1_{t-1}\otimes I_m
	\end{bmatrix}\right)
	.
	\qedhere
\end{equation*}
\end{proof}
According to \cref{thm:incorp:care}, we are able to make incorporation easily and solve \cref{eq:incorp}.
One thing worth mentioning is that $C_t$ is not difficult to calculate.
Putting \cref{eq:I+TT:inv:care} into \cref{eq:thm:incorp:care},
\begin{align*}
	C_t=C+
	\sqrt{2\gamma}
			\begin{bmatrix}
				\Xi_{1,I} \\  \Xi_{2,I}
			\end{bmatrix}^{\T}\begin{bmatrix}
				 (I_{t}\otimes Q_1)^{-1}&\\&(I_{t}\otimes \left[W+WY_\gamma ^{\T}Y_\gamma W\right])^{-1}
			\end{bmatrix}\begin{bmatrix}
				 \Xi_1 \\  \Xi_2
			\end{bmatrix},
\end{align*}
where $\Xi_1,\Xi_2$ is defined as in \cref{eq:Xt=SDS:care},
and
	\[
		\Xi_{1,I} = \toepU_{l\times l}\left(\begin{bmatrix}
				Q_2\\Q_1\\
		\end{bmatrix}\right)^{\T}(\bs 1_t\otimes I_l),
		\qquad
		\Xi_{2,I}=\toepU_{l\times m}\left(\begin{bmatrix}
				Q_3\\0\\
		\end{bmatrix}\right)^{\T}(\bs 1_t\otimes I_l)
		.
	\]

\Cref{thm:incorp:care} gives the detailed form of \cref{eq:incorp} at $\wtd X=X_t$.
Note that the sequence $\set{X_t}$ can be generated by the RADI method introduced in \cite{bennerBKS2018radi}, if the initial approximation is $0$ and the same shift $\gamma_t=\gamma$ is used at each step.
If we make incorporation at $t=1$ in each iteration, then the process is actually the RADI method.
As a direct consequence, we have the following result.
\begin{theorem}\label{thm:equiv:care}
For $X_t$ defined by DRE \cref{eq:ddre:care} and $C_t$ defined by \cref{eq:thm:incorp:care},
\[
	\Delta_t:=X_{t+1}-X_t
	=2\gamma(C_tA_{t,\gamma}^{-1})^{\T}(I+C_tA_{t,\gamma}^{-1}BB^{\T}A_{t,\gamma}^{-\T}C_t^{\T})^{-1}C_tA_{t,\gamma}^{-1},
	\quad 
	A_{t,\gamma}=A-BB^{\T}X_t-\gamma I,
\]
or equivalently, $\Delta_t$ is the approximate solution to \cref{eq:incorp} at $\wtd X=X_t$ generated by DRE \cref{eq:ddre:care} on $t=1$.
\end{theorem}

\section{Experiments and discussions}\label{sec:experiments-and-discussions}
In this section, we will provide several examples to illustrate the new algorithm FTA and compare it with some existing methods.
As is stated in \cref{ssec:fna-and-its-equivalence-to-many-other-methods}, many methods solve a CARE through an equivalent DARE.
Hence here we only test the CARE of the form
\[
	A^{\T}X+XA-XBB^{\T}X+C^{\T}C=0, \qquad A\in \mathbb{R}^{n\times n}, B\in \mathbb{R}^{n\times m}, C\in \mathbb{R}^{l\times n},
\]
for the performance of the methods on CARE can be recognized as those on DARE.
We will use these methods in the tests: 
\begin{itemize}
	\item FTA: our FFT-based Toeplitz-structured approximation with the incorporation technique;
	\item RKSM: rational Krylov subspace method \cite{guttel2013rational,druskinS2011adaptive,simonciniSM2014two};
	\item RADI+opt: RADI method \cite{bennerBKS2018radi}, with the residual minimizing shifts;
	\item RADI+proj: RADI method with the  residual Hamiltonian shifts;
	\item NK-ADI+GP: the Galerkin projected variant of Newton-Kleinman ADI method \cite{bennerHSW2016inexact,bennerLP2008numerical,bennerS2013numerical,bennerS2010newtongalerkinadi};
	\item iNK-ADI+LS: the inexact variant of Newton-Kleinman ADI method with line search.
\end{itemize}

All experiments are done in MATLAB 2021a under the Windows 10 Professional 64-bit operating system on a PC with a Intel Core i7-8700 processor at 3.20GHz and 64GB RAM. 
The implementation of RKSM comes from the source codes from Simoncini's  homepage\footnote{\url{http://www.dm.unibo.it/~simoncin/software.html}} with some modifications;
we use corresponding functions in the package M-M.E.S.S.\ version 2.1 \cite{SaaKB21-mmess-2.1} as the implementations of the last four methods.

The methods are intentionally chosen: RADI is the recommended method by the package M-M.E.S.S., and it is usually one of the fastest methods among the non-projective methods, and two different shift selection strategies are used, for there does not exist a definitely good one and both strategies are good in many tests;
the two variants of NK-ADI are Newton-type methods; RKSM is a projection method.
On the other hand, FTA, SDA, RADI, quadratic ADI and Cayley transformed Hamiltonian subspace iteration are theoretically equivalent if the shifts are the same;
the last three are of the same type, so only one of them, namely RADI, is chosen; SDA is appropriate for small-to-mid scale dense problems, so we give up putting it into comparison.

Since classical performance indices behave very different in different methods, we directly use the accuracy vs.\ the running time to compare.
The accuracy is measured by
\[
	\nres(X):=\frac{\N{\op C(X)}_{\F}}{\N{\op C(0)}_{\F}}=
	\frac{\N{A^{\T}X+XA-XBB^{\T}X+C^{\T}C}_{\F}}{\N{C^{\T}C}_{\F}}.
\]
In the following, three examples are tested, where the results are shown in \cref{fig:}.

For the parameters,
in FTA, we choose some $\gamma$ and use $t=64$ in one round, and then do incorporation with another $\gamma$ and $t=64$ until the convergence.
Since $\gamma$ is a shift on the matrix $A$, the choice of $\gamma$ with the same magnitude of the matrix $A$ would perform well.
In the following examples, we use a positive number consisting of one random digit and a hand-picked magnitude as $\gamma$. For instance,
$\gamma=8\times 10^{-7}$ is used in \cref{eg:rail}, where $8$ is random chosen and $-7$ is hand-picked (in fact picked according to the choice in SDA).
In our tests, the digit is not as important as the magnitude on the convergence speed;
for example, in \cref{eg:rail} FTA with $\gamma=5,6,7,8,9,10\times 10^{-7}$ performs nearly the same.
In addition, since FTA share the same theoretical convergence with SDA, the strategy of choosing $\gamma$ in SDA (see, e.g., \cite[Section~5.5]{huangLL2018structurepreserving}) should work in FTA, so should the sensitivity to $\gamma$.

	Each of the other five methods has its own way to choose shifts, so we leave the task for their own.
	Similar arguments apply for the following examples.

\begin{figure}[htp]
	\centering
	{\includegraphics[angle=90,width=0.415\linewidth]{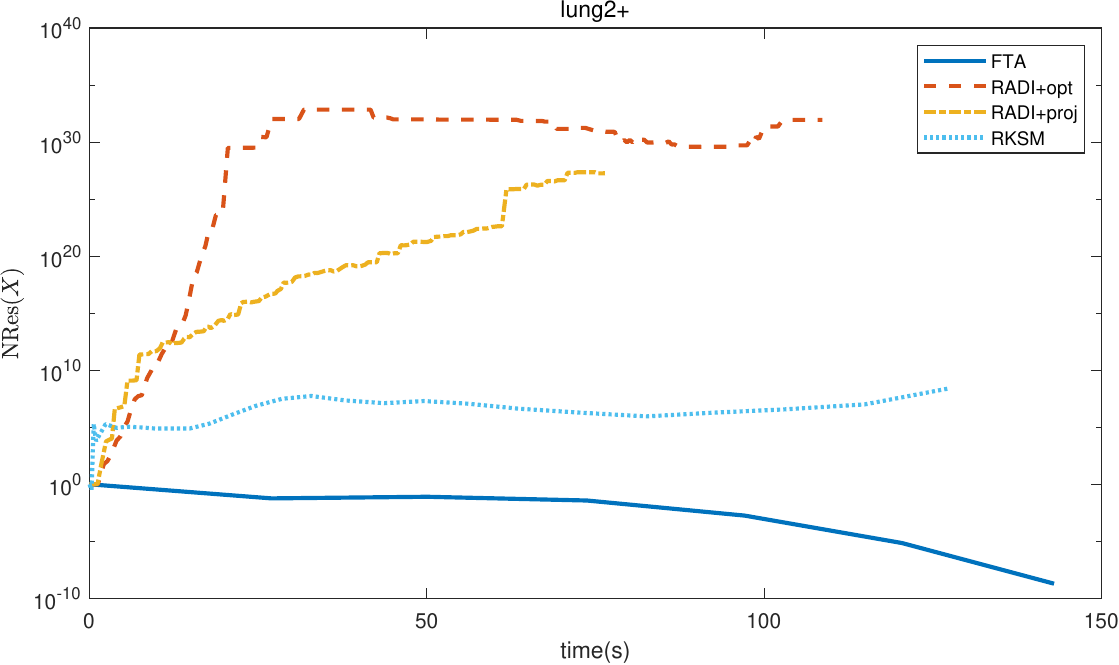}}
	{\includegraphics[angle=90,width=0.415\linewidth]{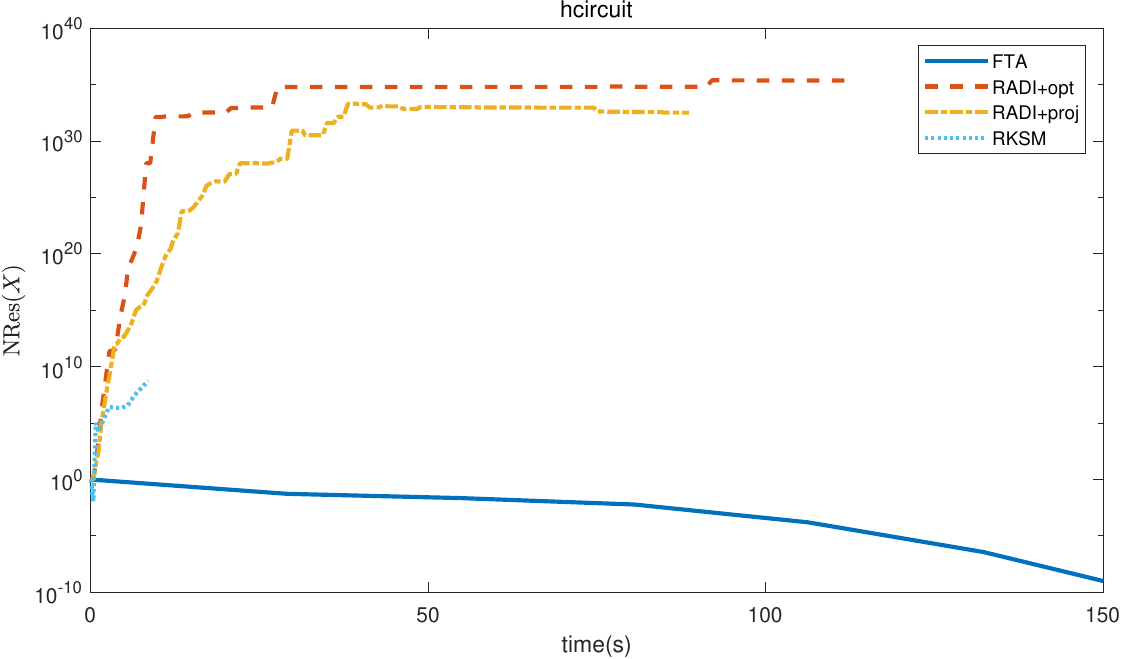}}

	{\includegraphics[angle=90,width=0.415\linewidth]{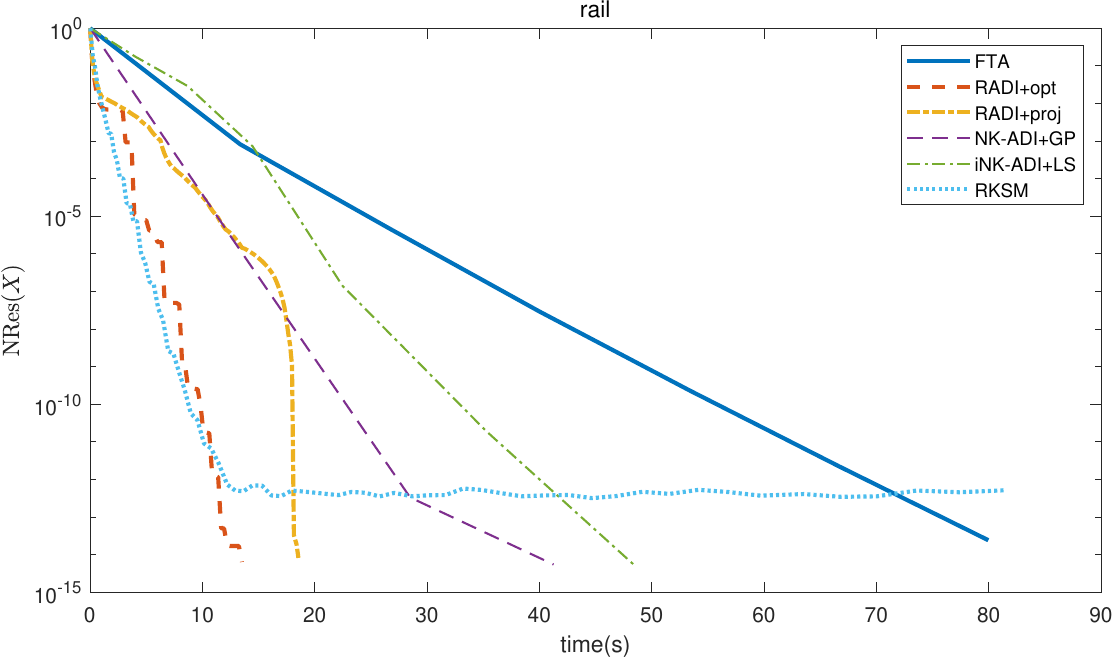}}
	{\includegraphics[angle=90,width=0.415\linewidth]{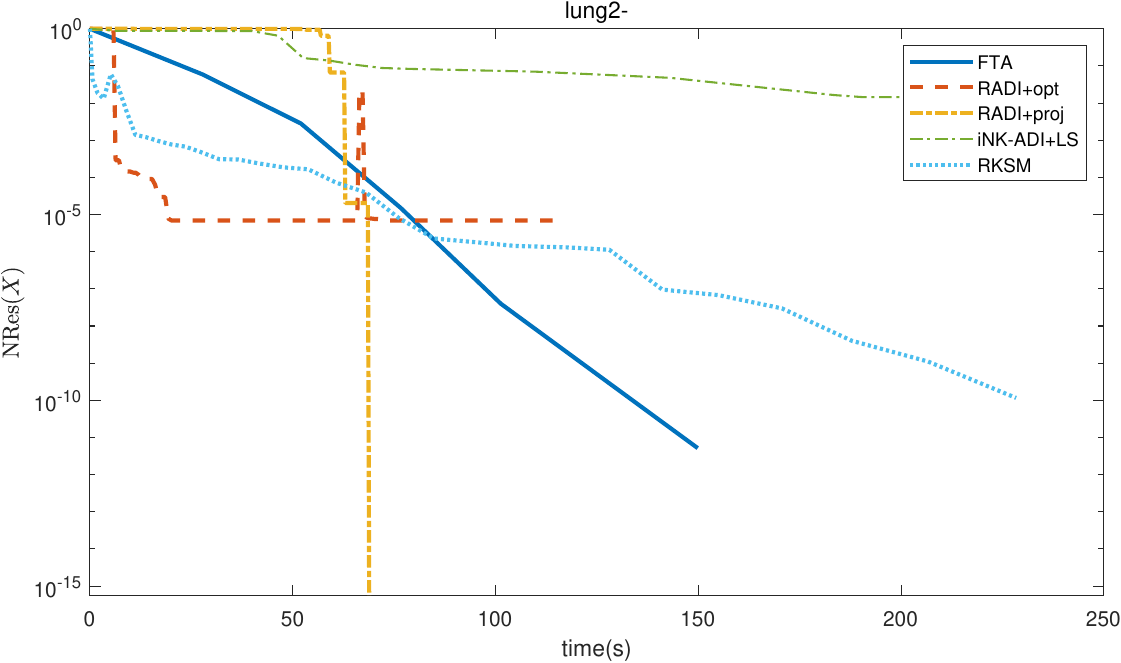}}

	\caption{accuracy vs.\ time}
	\label{fig:}
\end{figure}
\begin{example}[Rail]\label{eg:rail}
	The example is a version of the steel profile cooling model from the Oberwolfach Model Reduction Benchmark Collection, hosted at MORwiki \cite{morwiki_steel}.
	The data include $A\preceq0, E\succeq0, B,C$ with $n=79841,m=7,l=6$.
	Since we only focus on solving the CARE, $E$ is simply dropped.

	For the parameters, in FTA, we use a heuristic shift $\gamma=8\times 10^{-7}$ and in each incorporation step $\gamma\leftarrow \gamma/1.01$.

	In this example $A\prec 0$ and thus c-stable, which implies the properties of this problem are good. This results in the fact that all methods converge.
	We can see that the FTA is the slowest one among all the six methods.
	This phenomenon is reasonable. FTA and RADI are theoretically equivalent, while the only difference is that RADI has much more chances to choose different shifts to accelerate its convergence.
	Good shifts largely accelerate its convergence, and on the opposite, bad shifts would slow it down.
	RKSM and NK-ADI also benefit from the choice of shifts.
\end{example}
\begin{example}[Lung2$+$]\label{eg:lung2+}
	The example is generated in this way: $A$ is the matrix \texttt{lung2} in the SuiteSparse Matrix Collection \cite{davisH2011university} (formerly the University of Florida Sparse Matrix Collection), modelling temperature and water vapor transport in the human lung; $B,C$ are generated by MATLAB function \texttt{rand}.
	Here $n=109460,m=10,l=10$.

	For the parameters, in FTA, we use a heuristic shift $\gamma=5\times 10^{3}$ and in each incorporation step $\gamma\leftarrow \gamma/1.01$.

	In this example $A$ is nonsymmetric and the eigenvalues of $A$ lie in the right half plane, namely $A$ is c-anti-stable, or $-A$ is c-stable. 
	None of RADI+opt, RADI+proj, and RKSM converges.
	NK-ADI+GP and iNK-ADI+LS both report that non-stable Ritz values were detected and terminated the process.
	Only FTA produces a good approximate solution.
	This tells that the other five methods strongly rely on the stability of $A$.
	For example, a sufficient condition for achieving the convergence is that $A$ is stable and the shifts $\gamma_k$ satisfy the non-Blaschke condition
	$\sum_{k=1}^{+\infty}\frac{\Re(\gamma_k)}{1+\gamma_k\ol \gamma_k}=-\infty$ \cite{massoudiOR2016analysis}.
	However, the FTA works well even for the case that $A$ is not stable, which implies that in this sense the FTA is more robust with respect to the spectrum of $A$.
\end{example}
\begin{example}[Lung2$-$]\label{eg:lung2-}
	The example is almost the same with \cref{eg:lung2+} except that the matrix \texttt{lung2} is used as $-A$ rather than $A$. 

	For the FTA, we still use a heuristic shift $\gamma=5\times 10^{3}$ and in each incorporation step $\gamma\leftarrow \gamma/1.01$.

	In this example $A$ is nonsymmetric and c-stable.
	Note that the only difference between RADI+opt and RADI+proj is the different shift selection strategies.
	RADI+opt tends to converge fast but finally stays at a low accuracy;
	RADI+proj becomes convergent very late but soonly converges in a very short time.
	The phenomenon illustrates that the choice of shifts fatally affects its speed of convergence.
	The iNK-ADI+LS converges very slowly, while the NK-ADI+GP reports that non-stable Ritz values were detected again.
	The RKSM converges in a fairly good speed but finally slow down.
	The FTA converges steadily in a predictable speed.

	In another view, compared with \cref{eg:lung2+}, the running time of the FTA is nearly the same for different $A$'s, so the running time is predictable and can be estimated in advance.

\end{example}
\begin{example}[Hcircuit]\label{eg:hcircuit}
	The example is generated in this way: $A$ is the matrix \texttt{hcircuit} in the SuiteSparse Matrix Collection \cite{davisH2011university} (formerly the University of Florida Sparse Matrix Collection), modelling a circuit without parasitics; $B,C$ are generated by MATLAB function \texttt{rand}.
	Here $n=105676,m=10,l=10$.

	For the parameters, in FTA, we use a heuristic shift $\gamma=5\times 10^{3}$ and in each incorporation step $\gamma\leftarrow \gamma/1.01$.

	In this example $A$ is nonsymmetric and the real eigenvalues of $A$ lie in $[-1,86.3]$, namely $A$ is neither c-stable nor c-anti-stable, or equivalently neither of $\pm A$ is c-stable. 
	None of RADI+opt, RADI+proj, and RKSM converges,
	and RKSM terminiates in the midway, reporting that the projected Riccati equation does not have a finite solution.
	NK-ADI+GP and iNK-ADI+LS both report that non-stable Ritz values were detected and terminates the process.
	Only FTA produces a good approximate solution.
	This tells that the other five methods strongly rely on the stability of $A$, as is illustrated in \cref{eg:lung2+}.
\end{example}

Summarizing the numerical results, we see that the FTA has two significant features:
\begin{enumerate}
	\item FTA is robust in some sense and it converges no matter how the property of $A$ is;
	\item FTA has a steady convergence rate and the execution time is predictable, although in good cases it may converge slowly compared with other methods.
\end{enumerate}
Moreover, it is easy to see that if $A$ is dense, FTA needs the LU/PLU factorization only several times, while the other methods need as many as number of iterations, according to the number of used shifts.
\section{Conclusion}\label{sec:conclusions}
We have presented our FFT-based Toeplitz-structured approximation method for computing the stabilizing solution of large-scale algebraic Riccati equations with low-rank structure.
It is shown that the closed form given by operator theory under good assumptions is also valid for the general case, which is proved by matrix analysis.
It is quite natural to ask  whether the closed form can be directly produced by the analysis of unbounded linear operators, which would be a difficult task for future work.
On the numerical front, our method works robust in some sense and few parameters are needed.
As the readers may see, there is still possibility to improve the behavior by adopting more techniques. However, to keep this paper compact and concentrated, we leave it for another work.

\appendix
\section{
Displacement rank and Toeplitz matrix}\label{ssec:displacement-rank-and-toeplitz-matrix}

In order to prove \cref{lm:form-of-inverse:dare},
we first give a few results to the displacement rank and Toeplitz matrices,
and interested readers are referred to the review paper~\cite{kailathS1995displacement} and the references therein. 

For any matrix $R\in \mathbb{R}^{pn\times pn}$, its \emph{$(\pm)$-displacement rank} $\alpha_{\pm}(R,p)$ with respect to block size $p\times p$, is defined by 
\[
	\alpha_+(R,p):=\rank_p(R-Z_{n,p}RZ_{n,p}^{\T}),\qquad
	\alpha_-(R,p):=\rank_p(R-Z_{n,p}^{\T}RZ_{n,p}),\qquad
\]
 where $Z_{n,p}=\begin{bmatrix}
	0 & 0 \\ I_{(n-1)p} & 0
\end{bmatrix}_{pn\times pn}$, and $\rank_p(\cdot)$ is considered as the rank of the linear transformation on the module $\mathbb{R}^{np\times p}$ over the ring $\mathbb{R}^{p\times p}$. For the case $p=1$, $\rank_p(\cdot)=\rank(\cdot)$, the ordinary rank of matrices in $\mathbb{R}^{n\times n}$.

The definition is based on the following result, namely \cref{lm:equiv-expression}.
\begin{lemma}[{\cite{kailathKM1979displacement}}]\label{lm:equiv-expression}
\begin{subequations}\label{eq:lm:equiv-expression}
	Given $R_1,R_2\in \mathbb{R}^{pn\times p}$ and $R\in \mathbb{R}^{pn\times pn}$, then
	\begin{align*}
		R-Z_{n,p}RZ_{n,p}^{\T} = R_1R_2^{\T} &\iff R=\toepL_{p\times p}(R_1)\toepL_{p\times p}(R_2)^{\T},
		\\
		R-Z_{n,p}^{\T}RZ_{n,p} = R_1R_2^{\T} &\iff R=\toepU_{p\times p}(R_1)\toepU_{p\times p}(R_2)^{\T}.
	\end{align*}
\end{subequations}
\end{lemma}
\Cref{lm:equiv-expression} implies that for a matrix its displacement rank is related to how it can be expressed as a sum of products of block-Toeplitz matrices, as is shown in \cref{lm:displacement-rank}.
\begin{lemma}[{\cite{kailathKM1979displacement,kailathC1994generalized,kailathS1995displacement}}]\label{lm:displacement-rank}
	Given a matrix $R\in \mathbb{R}^{pn\times pn}$.
	\begin{subequations}\label{eq:lm:displacement-rank}
		\begin{enumerate}
			\item Its $(+)$-displacement rank $\alpha_+(R,p)$ is the smallest integer $\beta$ such that $R$ can be written in the form
				\begin{equation}\label{eq:lm:displacement-rank:+}
					R=\sum_{i=1}^{\beta}\toepL_{p\times p}(R_{i})\toepU_{p\times p}(\wtd R_{i}),
				\end{equation}
				where $R_{i},\wtd R_{i}\in \mathbb{R}^{pn\times p}$.
			\item Its $(-)$-displacement rank $\alpha_-(R,p)$ is the smallest integer $\beta$ such that $R$ can be written in the form
				\begin{equation}\label{eq:lm:displacement-rank:-}
					R=\sum_{i=1}^{\beta}\toepU_{p\times p}(R_{i})\toepL_{p\times p}(\wtd R_{i}),
				\end{equation}
				where $R_i,\wtd R_i\in \mathbb{R}^{pn\times p}$.
			\item If $R$ is symmetric and positive semidefinite, \cref{eq:lm:displacement-rank:+,eq:lm:displacement-rank:-} can be replaced respectively by
				\[
					R=\sum_{i=1}^{\beta}\toepL_{p\times p}(R_{i})\toepL_{p\times p}(R_{i})^{\T},\qquad \text{and} \qquad
					R=\sum_{i=1}^{\beta}\toepU_{p\times p}(R_{i})\toepU_{p\times p}(R_{i})^{\T}.
				\]

			\item If $R$ is nonsingular, then $\alpha_+(R,p)=\alpha_-(R^{-1},p),\alpha_-(R,p)=\alpha_+(R^{-1},p)$.
		\end{enumerate}
	\end{subequations}
\end{lemma}
\Cref{lm:displacement-rank} demonstrates the relation between the displacement ranks of a matrix and its inverse, which is actually the theoretical foundation of the fast and superfast algorithms on Toeplitz matrices.

The following result, namely \cref{lm:form-of-inverse:+2}, gives an expression of the inverse related to the displacement rank.
\begin{lemma}[{\cite{friedlanderMKL1979new}}]\label{lm:form-of-inverse:+2}
	Given $R\in \mathbb{R}^{pn\times pn}$, suppose
	\begin{enumerate}
		\item $R$ is nonsingular, and
			$ R^{-1}=\begin{bmatrix}
				Q_{1,t} & Q^L_1\\
				Q_1 & *\\
				\end{bmatrix}=\begin{bmatrix}
				* & Q_2 \\ 
				Q^L_2 & Q_{2,b}\\
			\end{bmatrix}$ where $Q_{1,t},Q_{2,b}\in \mathbb{R}^{p\times p}$ are nonsingular;
		\item 
			$R-Z_{n,p}RZ_{n,p}^{\T}=\begin{bmatrix}
				* & * \\ * & D_1\Sigma D_2^{\T}
			\end{bmatrix}$ where $D_1,D_2\in \mathbb{R}^{p(n-1)\times p\alpha}$ and $\Sigma
			$ is a diagonal matrix whose diagonal entries are $\pm1$;
		\item writing $R=\begin{bmatrix}
				* & * \\
				* & R_s\\
			\end{bmatrix}$ where $R_s\in \mathbb{R}^{p(n-1)\times p(n-1)}$, there exist $Q_3,Q^L_3\in \mathbb{R}^{p(n-1)\times p\alpha}$ such that $R_sQ_3=D_1,Q^L_3R_s=D_2^{\T}$.
	\end{enumerate}
	Then
	\begin{subequations}\label{eq:lm:form-of-inverse:+2}
		\begin{align}
\label{eq:lm:form-of-inverse:+2:-}
			R^{-1}&=
			-\toepU_{p\times p}\left(\begin{bmatrix}
					Q_1\\0_{p\times p}\\
					\end{bmatrix}\right)(I_{n}\otimes Q_{1,t})^{-1}\toepU_{p\times p}\left(\begin{bmatrix}
					Q^L_1&0_{p\times p}
			\end{bmatrix}^{\T}\right)^{\T}
			\nonumber\\&\qquad	+\toepU_{p\times p}\left(\begin{bmatrix}
					Q_2\\Q_{2,b}\\
					\end{bmatrix}\right)(I_{n}\otimes Q_{2,b})^{-1}\toepU_{p\times p}\left(\begin{bmatrix}
					Q^L_2&Q_{2,b}\\
			\end{bmatrix}^{\T}\right)^{\T}
			\nonumber\\&\qquad	+\toepU_{p\times p\alpha }\left(\begin{bmatrix}
					Q_3\\0_{p\times p\alpha}\\
					\end{bmatrix}\right)(I_{n}\otimes W)^{-1}\toepU_{p\times p\alpha}\left(\begin{bmatrix}
					Q^L_3&0_{p\alpha\times p}\\
			\end{bmatrix}^{\T}\right)^{\T}
			,
			\intertext{or alternatively,}
\label{eq:lm:form-of-inverse:+2:+}
			R^{-1}&=
			\toepL_{p\times p}\left(\begin{bmatrix}
					Q_{1,t}\\Q_1\\
					\end{bmatrix}\right)(I_{n}\otimes Q_{1,t})^{-1}\toepL_{p\times p}\left(\begin{bmatrix}
					Q_{1,t}&Q^L_1
			\end{bmatrix}^{\T}\right)^{\T}
			\nonumber\\&\qquad	-\toepL_{p\times p}\left(\begin{bmatrix}
					0_{p\times p}\\Q_2\\
					\end{bmatrix}\right)(I_{n}\otimes Q_{2,b})^{-1}\toepL_{p\times p}\left(\begin{bmatrix}
					0_{p\times p}&Q^L_2\\
			\end{bmatrix}^{\T}\right)^{\T}
			\nonumber\\&\qquad	-\toepL_{p\times p\alpha}\left(\begin{bmatrix}
					0_{p\times p\alpha}\\Q_3\\
					\end{bmatrix}\right)(I_{n}\otimes W)^{-1}\toepL_{p\times p\alpha}\left(\begin{bmatrix}
					0_{p\alpha\times p}&Q^L_3\\
			\end{bmatrix}^{\T}\right)^{\T}
			,
		\end{align}
	\end{subequations}
	where $W=\Sigma-Q^L_3D_1$.

	Moreover, if $R$ is symmetric, then there exists a factorization to make $D_1=D_2$; for that case, \cref{eq:lm:form-of-inverse:+2} can be rewritten by $Q^L_1=Q_1^{\T},Q^L_2=Q_2^{\T},Q^L_3=Q_3^{\T}$.
\end{lemma}
\begin{remark}\label{rk:lm:form-of-inverse:+2}
	Item~2 of \cref{lm:form-of-inverse:+2} implies that $R-Z_{n,p}^{\T}RZ_{n,p}=\begin{bmatrix}
		-D_1\Sigma D_2^{\T} & *\\ * &*
	\end{bmatrix}$.
\end{remark}
Note that \cref{eq:lm:form-of-inverse:+2} presents a sum of $\alpha+2$ products of block-Toeplitz matrices, in which the number of terms may not be the smallest one, namely $\alpha_{\mp}(R,p)$.

\bigskip
In the following, we will derive a sum of the $\alpha_+(R,p)=\alpha_-(R^{-1},p)$ terms, called  a \emph{shortest sum}, to coincide with \cref{lm:displacement-rank}.
Using the same way a sum of $\alpha_-(R,p)=\alpha_+(R^{-1},p)$ terms can also be derived, so we omit the details.

Write $R=\begin{bmatrix}
	R_{11} & R_{12} \\ R_{21} & R_s\\
	\end{bmatrix}$, and then $R-Z_{n,p}RZ_{n,p}^{\T}=\begin{bmatrix}
	R_{11} & R_{12}\\ R_{21} & D_1\Sigma D_2^{\T}\\
\end{bmatrix}$.
Thus, $\alpha\le\alpha_+(R,p)\le \alpha+2$,
provided that $\rank(D_1\Sigma D_2^T)=\rank(\Sigma)=p\alpha$.

	On the other hand, by \cref{eq:smwf}, under sufficient nonsingular conditions, it is easy to have
\begin{align*}
	\begin{bmatrix}
		R_{11} & R_{12} \\ R_{21} & R_s \\
	\end{bmatrix}^{-1}
	&=
	\begin{bmatrix}
		R_{11}^{-1}+R_{11}^{-1}R_{12}(R_s-R_{21}R_{11}^{-1}R_{12})^{-1}R_{21}R_{11}^{-1} 
		& -R_{11}^{-1}R_{12}(R_s-R_{21}R_{11}^{-1}R_{12})^{-1}\\
		-(R_s-R_{21}R_{11}^{-1}R_{12})^{-1}R_{21}R_{11}^{-1}
		&(R_s-R_{21}R_{11}^{-1}R_{12})^{-1} \\
	\end{bmatrix}
	\\&=\begin{bmatrix}
		(R_{11}-R_{12}R_s^{-1}R_{21})^{-1}
		&-(R_{11}-R_{12}R_s^{-1}R_{21})^{-1}R_{12}R_s^{-1}\\
		-R_s^{-1}R_{21}(R_{11}-R_{12}R_s^{-1}R_{21})^{-1}
		&
		R_s^{-1}+R_s^{-1}R_{21}(R_{11}-R_{12}R_s^{-1}R_{21})^{-1}R_{12}R_s^{-1}\\
	\end{bmatrix}
	.
\end{align*}
Compared with the conditions,
\begin{align*}
	Q_1 = -R_s^{-1}R_{21}Q_{1,t},\quad
	Q^L_1 = -Q_{1,t}R_{12}R_s^{-1},\quad
	Q_{1,t} = (R_{11}-R_{12}R_s^{-1}R_{21})^{-1}.
\end{align*}

If $\alpha_+(R,p)=\alpha$, then it has to hold that $R-Z_{n,p}RZ_{n,p}^{\T}=\begin{bmatrix}
	S_1^{\T}\Sigma^{-1}S_2 & S_1^{\T} D_2^{\T}\\ D_1S_2 & D_1\Sigma D_2^{\T}\\
\end{bmatrix}$ for some $S_1,S_2\in \mathbb{R}^{p\alpha \times p}$.
Clearly $S_1,S_2$ are of full column rank for $R$ is nonsingular.
Noticing $\Sigma^{-1}=\Sigma$,
we have
\begin{align*}
	Q_1 &= -R_s^{-1}D_1S_2Q_{1,t}=-Q_3S_2Q_{1,t},\quad\\
	Q^L_1 &= -Q_{1,t}S_1^{\T}D_2^{\T}R_s^{-1}=-Q_{1,t}S_1^{\T}Q^L_3,\quad\\
	Q_{1,t} &= (S_1^{\T}\Sigma^{-1}S_2-S_1^{\T}D_2^{\T}R_s^{-1}D_1S_2)^{-1}
	= (S_1^{\T}\Sigma^{-1}S_2-S_1^{\T}Q^L_3D_1S_2)^{-1}=(S_1^{\T}WS_2)^{-1}.
\end{align*}
Thus
\begin{align*}
\MoveEqLeft \toepU_{p\times p}\left(\begin{bmatrix}
					Q_1\\0\\
					\end{bmatrix}\right)(I_{n}\otimes Q_{1,t})^{-1}\toepU_{p\times p}\left(\begin{bmatrix}
					Q^L_1&0
			\end{bmatrix}^{\T}\right)^{\T}
	\\&=\toepU_{p\times p}\left(\begin{bmatrix}
		-Q_3S_2Q_{1,t}\\0\\
			\end{bmatrix}\right)(I_{n}\otimes Q_{1,t})^{-1}\toepU_{p\times p}\left(\begin{bmatrix}
			-Q_{1,t}S_1^{\T}Q^L_3&0\\
	\end{bmatrix}^{\T}\right)^{\T}
	\\&=\toepU_{p\times p\alpha}\left(\begin{bmatrix}
		Q_3\\0\\
			\end{bmatrix}\right)(I_{n}\otimes S_2Q_{1,t})(I_{n}\otimes Q_{1,t})^{-1}(I_{n}\otimes Q_{1,t}S_1^{\T})\toepU_{p\times p\alpha}\left(\begin{bmatrix}
			Q^L_3&0\\
	\end{bmatrix}^{\T}\right)^{\T}
	\\&=\toepU_{p\times p\alpha}\left(\begin{bmatrix}
		Q_3\\0\\
			\end{bmatrix}\right)(I_{n}\otimes S_2Q_{1,t}S_1^{\T})\toepU_{p\times p\alpha}\left(\begin{bmatrix}
			Q^L_3&0\\
	\end{bmatrix}^{\T}\right)^{\T}
	\\&=\toepU_{p\times p\alpha}\left(\begin{bmatrix}
		Q_3\\0\\
		\end{bmatrix}\right)\left(I_{n}\otimes S_2(S_1^{\T}WS_2)^{-1}S_1^{\T}\right)\toepU_{p\times p\alpha}\left(\begin{bmatrix}
			Q^L_3&0\\
	\end{bmatrix}^{\T}\right)^{\T}
	.
\end{align*}
Note that 
\begin{align*}
	\left[W^{-1}-S_2\left(S_1^{\T}WS_2\right)^{-1}S_1^{\T}\right]WS_2
	=0
	.
\end{align*}
Complement $S_2$ to a nonsingular matrix $\begin{bmatrix}
	S_2 & S_2^{c}
\end{bmatrix}$, and then
\begin{align*}
	\left[W^{-1}-S_2\left(S_1^{\T}WS_2\right)^{-1}S_1^{\T}\right]W\begin{bmatrix}
			S_2 & S_2^{c}
		\end{bmatrix}
		&=\begin{bmatrix}
			S_2 & S_2^{c}
		\end{bmatrix}\begin{bmatrix}
		0_{p\times p} & -\left(S_1^{\T}WS_2\right)^{-1}S_1^{\T}WS_2^{c}\\ 0 & I_{p(\alpha-1)}\\
		\end{bmatrix}
		\\&=\begin{bmatrix}
			S_2 & S_2^{c}
		\end{bmatrix}\begin{bmatrix}
		-\left(S_1^{\T}WS_2\right)^{-1}S_1^{\T}WS_2^{c}\\  I_{p(\alpha-1)}\\
		\end{bmatrix}\begin{bmatrix}
		0 & I_{p(\alpha-1)}
		\end{bmatrix},
\end{align*}
whose rank is $p(\alpha-1)$.
Write
\begin{align*}
	W_1&=\begin{bmatrix}
		S_2 & S_2^{c}
		\end{bmatrix}\begin{bmatrix}
		-\left(S_1^{\T}WS_2\right)^{-1}S_1^{\T}WS_2^{c}\\  I_{p(\alpha-1)}
	\end{bmatrix}\in \mathbb{R}^{p\alpha\times p(\alpha-1)},
	\qquad
	\\
	W^L_1&=\begin{bmatrix}
		0 & I_{p(\alpha-1)}
		\end{bmatrix}\begin{bmatrix}
		S_2 & S_2^{c}
	\end{bmatrix}^{-1}W^{-1}\in \mathbb{R}^{p(\alpha-1)\times p\alpha},
\end{align*}
	and then $W^{-1}-S_2\left(S_1^{\T}WS_2\right)^{-1}S_1^{\T}=W_1W^L_1$.
Hence 
\begin{equation}\label{eq:lm:form-of-inverse:0}
	\begin{aligned}[b]
				R^{-1}
				&=
				\toepU_{p\times p}\left(\begin{bmatrix}
						Q_2\\Q_{2,b}\\
						\end{bmatrix}\right)(I_{n}\otimes Q_{2,b})^{-1}\toepU_{p\times p}\left(\begin{bmatrix}
						Q^L_2&Q_{2,b}\\
				\end{bmatrix}^{\T}\right)^{\T}
				 +\toepU_{p\times p\alpha}\left(\begin{bmatrix}
						Q_3\\0\\
						\end{bmatrix}\right)(I_{n}\otimes W_1W^L_1)\toepU_{p\times p\alpha}\left(\begin{bmatrix}
						Q^L_3&0\\
				\end{bmatrix}^{\T}\right)^{\T}
				\\&=
				\toepU_{p\times p}\left(\begin{bmatrix}
						Q_2\\Q_{2,b}\\
						\end{bmatrix}\right)(I_{n}\otimes Q_{2,b})^{-1}\toepU_{p\times p}\left(\begin{bmatrix}
						Q^L_2&Q_{2,b}\\
				\end{bmatrix}^{\T}\right)^{\T}
				+\toepU_{p\times p(\alpha-1)}\left(\begin{bmatrix}
						Q_3W_1\\0\\
						\end{bmatrix}\right)\toepU_{p\times p(\alpha-1)}\left(\begin{bmatrix}
						W^L_1Q^L_3&0\\
				\end{bmatrix}^{\T}\right)^{\T}
		.	
		\end{aligned}
\end{equation}

If $\alpha_+(R,p)=\alpha+1$, then it holds that 
\begin{subequations}\label{eq:a_+=a+1}
	\begin{align}
		\label{eq:a_+=a+1:1}
	R-Z_{n,p}RZ_{n,p}^{\T}
	&=\begin{bmatrix}
			S_1^{\T}\Sigma^{-1}S_2 & S_1^{\T} D_2^{\T}\\ D_1S_2 & D_1\Sigma D_2^{\T}\\
			\end{bmatrix}+\begin{bmatrix}
			S_3 & D_3^{\T}\\0 & 0
			\end{bmatrix}
			\\
			&\quad\text{or}\quad \begin{bmatrix}
			S_1^{\T}\Sigma^{-1}S_2 & S_1^{\T} D_2^{\T}\\ D_1S_2 & D_1\Sigma D_2^{\T}\\
			\end{bmatrix}+\begin{bmatrix}
			S_3 & 0\\ D_3 & 0
		\end{bmatrix}
		\label{eq:a_+=a+1:2}
	\end{align}
\end{subequations}
for some $S_1,S_2\in \mathbb{R}^{p\alpha\times p}$, $S_3\in \mathbb{R}^{p\times p}$ and $D_3\in \mathbb{R}^{p(n-1)\times p}$.

Consider \cref{eq:a_+=a+1:1}.
Then,
\begin{align*}
	Q_1 &= -R_s^{-1}D_1S_2Q_{1,t}=-Q_3S_2Q_{1,t},\quad\\
	Q^L_1 &= -Q_{1,t}(S_1^{\T}D_2^{\T}+D_3^{\T})R_s^{-1}=-Q_{1,t}S_1^{\T}Q^L_3-Q_{1,t}D_3^{\T}R_s^{-1},\quad\\
	Q_{1,t} &= (S_1^{\T}\Sigma^{-1}S_2+S_3-(S_1^{\T}D_2^{\T}+D_3^{\T})R_s^{-1}D_1S_2)^{-1}
	\\&= (S_1^{\T}\Sigma^{-1}S_2+S_3-D_3^{\T}Q_3S_2-S_1^{\T}Q^L_3D_1S_2)^{-1}
	\\&=(S_3-D_3^{\T}Q_3S_2+S_1^{\T}WS_2)^{-1}.
\end{align*}
Thus,
\begin{align*}
\MoveEqLeft \toepU_{p\times p}\left(\begin{bmatrix}
					Q_1\\0\\
					\end{bmatrix}\right)(I_{n}\otimes Q_{1,t})^{-1}\toepU_{p\times p}\left(\begin{bmatrix}
					Q^L_1&0
			\end{bmatrix}^{\T}\right)^{\T}
	\\&=\toepU_{p\times p}\left(\begin{bmatrix}
		-Q_3S_2Q_{1,t}\\0\\
			\end{bmatrix}\right)(I_{n}\otimes Q_{1,t})^{-1}\toepU_{p\times p}\left(\begin{bmatrix}
-Q_{1,t}S_1^{\T}Q^L_3-Q_{1,t}D_3^{\T}R_s^{-1}&0\\
	\end{bmatrix}^{\T}\right)^{\T}
	\\&=
	\begin{multlined}[t]
	\toepU_{p\times p\alpha}\left(\begin{bmatrix}
		Q_3\\0\\
			\end{bmatrix}\right)(I_{n}\otimes S_2Q_{1,t})(I_{n}\otimes Q_{1,t})^{-1}(I_{n}\otimes Q_{1,t}S_1^{\T})\toepU_{p\times p\alpha}\left(\begin{bmatrix}
			Q^L_3&0\\
	\end{bmatrix}^{\T}\right)^{\T}
	\\+ \toepU_{p\times p\alpha}\left(\begin{bmatrix}
		Q_3\\0\\
			\end{bmatrix}\right)(I_{n}\otimes S_2Q_{1,t})(I_{n}\otimes Q_{1,t})^{-1}(I_{n}\otimes Q_{1,t})\toepU_{p\times p}\left(\begin{bmatrix}
			D_3^{\T}R_s^{-1}&0\\
	\end{bmatrix}^{\T}\right)^{\T}
	\end{multlined}
	\\&=
	\begin{multlined}[t]
		\toepU_{p\times p\alpha}\left(\begin{bmatrix}
		Q_3\\0\\
			\end{bmatrix}\right)(I_{n}\otimes S_2Q_{1,t}S_1^{\T})\toepU_{p\times p\alpha}\left(\begin{bmatrix}
			Q^L_3&0\\
	\end{bmatrix}^{\T}\right)^{\T}
	\\+ \toepU_{p\times p\alpha}\left(\begin{bmatrix}
		Q_3\\0\\
			\end{bmatrix}\right)(I_{n}\otimes S_2Q_{1,t})\toepU_{p\times p}\left(\begin{bmatrix}
			D_3^{\T}R_s^{-1}&0\\
	\end{bmatrix}^{\T}\right)^{\T}
	\end{multlined}
	\\&=
	\begin{multlined}[t]
		\toepU_{p\times p\alpha}\left(\begin{bmatrix}
		Q_3\\0\\
			\end{bmatrix}\right)(I_{n}\otimes S_2(S_3-D_3^{\T}Q_3S_2+S_1^{\T}WS_2)^{-1}S_1^{\T})\toepU_{p\times p\alpha}\left(\begin{bmatrix}
			Q^L_3&0\\
	\end{bmatrix}^{\T}\right)^{\T}
	\\+ \toepU_{p\times p\alpha}\left(\begin{bmatrix}
		Q_3\\0\\
			\end{bmatrix}\right)(I_{n}\otimes S_2Q_{1,t})\toepU_{p\times p}\left(\begin{bmatrix}
			D_3^{\T}R_s^{-1}&0\\
	\end{bmatrix}^{\T}\right)^{\T}
	.
	\end{multlined}
\end{align*}
Since
\begin{align*}
	W^{-1}-S_2(S_3-D_3^{\T}Q_3S_2+S_1^{\T}WS_2)^{-1}S_1^{\T}
	&=W^{-1}\left(W-WS_2(S_3-D_3^{\T}Q_3S_2+S_1^{\T}WS_2)^{-1}S_1^{\T}W\right)W^{-1}
	\\&\clue{\cref{eq:smwf}}{=}
	W^{-1}\left(W^{-1}+S_2(S_3-D_3^{\T}Q_3S_2)^{-1}S_1^{\T}\right)^{-1}W^{-1}
	\\&=
	\left(W+WS_2(S_3-D_3^{\T}Q_3S_2)^{-1}S_1^{\T}W\right)^{-1}
	=:W_1^{-1}
	,
\end{align*}
we have
\begin{equation}\label{eq:lm:form-of-inverse:1}
	\begin{aligned}[b]
				R^{-1}
				&=
				\begin{multlined}[t]
				\toepU_{p\times p}\left(\begin{bmatrix}
						Q_2\\Q_{2,b}\\
						\end{bmatrix}\right)(I_{n}\otimes Q_{2,b})^{-1}\toepU_{p\times p}\left(\begin{bmatrix}
						Q^L_2&Q_{2,b}\\
				\end{bmatrix}^{\T}\right)^{\T}
				\\+\toepU_{p\times p\alpha}\left(\begin{bmatrix}
						Q_3\\0\\
						\end{bmatrix}\right)(I_{n}\otimes W_1)^{-1}\toepU_{p\times p\alpha}\left(\begin{bmatrix}
						Q^L_3&0\\
				\end{bmatrix}^{\T}\right)^{\T}
	- \toepU_{p\times p\alpha}\left(\begin{bmatrix}
		Q_3\\0\\
			\end{bmatrix}\right)(I_{n}\otimes S_2Q_{1,t})\toepU_{p\times p}\left(\begin{bmatrix}
			D_3^{\T}R_s^{-1}&0\\
	\end{bmatrix}^{\T}\right)^{\T}
	\end{multlined}
				\\&=
				\begin{multlined}[t]
				\toepU_{p\times p}\left(\begin{bmatrix}
						Q_2\\Q_{2,b}\\
						\end{bmatrix}\right)(I_{n}\otimes Q_{2,b})^{-1}\toepU_{p\times p}\left(\begin{bmatrix}
						Q^L_2&Q_{2,b}\\
				\end{bmatrix}^{\T}\right)^{\T}
				\\+\toepU_{p\times p\alpha}\left(\begin{bmatrix}
						Q_3\\0\\
						\end{bmatrix}\right)(I_{n}\otimes W_1)^{-1}\toepU_{p\times p\alpha}\left(\begin{bmatrix}
						Q^L_3-W_1S_2Q_{1,t}D_3^{\T}R_s^{-1}&0\\
				\end{bmatrix}^{\T}\right)^{\T}
		.\end{multlined}
		\end{aligned}
\end{equation}
Similarly, for \cref{eq:a_+=a+1:2},
\begin{equation}\label{eq:lm:form-of-inverse:1'}
	\begin{aligned}[b]
				R^{-1}
				&=
				\begin{multlined}[t]
				\toepU_{p\times p}\left(\begin{bmatrix}
						Q_2\\Q_{2,b}\\
						\end{bmatrix}\right)(I_{n}\otimes Q_{2,b})^{-1}\toepU_{p\times p}\left(\begin{bmatrix}
						Q^L_2&Q_{2,b}\\
				\end{bmatrix}^{\T}\right)^{\T}
				\\+\toepU_{p\times p\alpha}\left(\begin{bmatrix}
						Q_3-R_s^{-1}D_3Q_{1,t}S_1^{\T}W_1\\0\\
						\end{bmatrix}\right)(I_{n}\otimes W_1)^{-1}\toepU_{p\times p\alpha}\left(\begin{bmatrix}
						Q^L_3&0\\
				\end{bmatrix}^{\T}\right)^{\T}
		,\end{multlined}
		\end{aligned}
\end{equation}
where $W_1=W+WS_2(S_3-S_1^{\T}Q^L_3D_3)^{-1}S_1^{\T}W$.

To sum up, we have \cref{lm:form-of-inverse}.
\begin{lemma}\label{lm:form-of-inverse}
	Given $R\in \mathbb{R}^{pn\times pn}$, suppose the conditions in \cref{lm:form-of-inverse:+2} hold.
	Then $\alpha \le \alpha_+(R,p)\le \alpha+2$, and the following statements hold.
	\begin{enumerate}
		\item if $\alpha_+(R,p)=\alpha$, then \cref{eq:lm:form-of-inverse:0} is a shortest sum.
		\item if $\alpha_+(R,p)=\alpha+1$, then \cref{eq:lm:form-of-inverse:1} or \cref{eq:lm:form-of-inverse:1'} is a shortest sum.
		\item if $\alpha_+(R,p)=\alpha+2$, then \cref{eq:lm:form-of-inverse:+2:-} is a shortest sum.
	\end{enumerate}

	Moreover, if $R$ is symmetric, then there exists a factorization to make $D_1=D_2$; for that case, \cref{eq:lm:form-of-inverse:+2:-,eq:lm:form-of-inverse:0,eq:lm:form-of-inverse:1,eq:lm:form-of-inverse:1'} can be rewritten by $Q^L_1=Q_1^{\T},Q^L_2=Q_2^{\T},Q^L_3=Q_3^{\T},S_1=S_2,D_3=0$.
\end{lemma}

Then we devote \cref{lm:form-of-inverse:practical}.
\begin{lemma}\label{lm:form-of-inverse:practical}
	Given $Y\in \mathbb{R}^{p_1\times p_2}$, $Y^L\in \mathbb{R}^{p_2\times p_1}$, $D_{t-1}\in \mathbb{R}^{p_1(t-1)\times p_2}$, $D^L_{t-1}\in \mathbb{R}^{p_2\times p_1(t-1)}$, let 
	\begin{align*}
		T_t
		&=\toepL_{p_1\times p_2}\left(\begin{bmatrix}
				Y \\ D_{t-1}\\
			\end{bmatrix}\right)
			=\begin{bmatrix}
				Y  & 0\\ D_{t-1} & T_{t-1}
			\end{bmatrix}\in \mathbb{R}^{p_1t\times p_2t}
			,\quad
			\\
			T^L_t
			&=\toepL_{p_1\times p_2}\left(\begin{bmatrix}
					Y^L& D^L_{t-1}\\
			\end{bmatrix}^{\T}\right)^{\T}
			=\begin{bmatrix}
				Y^L & D^L_{t-1}\\ 0 & T^L_{t-1}
			\end{bmatrix}\in \mathbb{R}^{p_2t\times p_1t}
			.
	\end{align*}
	If $I_{p_1t}-T_tT^L_t$ is nonsingular,
\begin{subequations}\label{eq:lm:form-of-inverse:practical}
then
\begin{equation}
			(I_{p_1t}-T_tT^L_t)^{-1}=
			\begin{multlined}[t]
			\toepU_{p_1\times p_1}\left(\begin{bmatrix}
					Q_2\\Q_{2,b}\\
					\end{bmatrix}\right)(I_{t}\otimes Q_{2,b})^{-1}\toepU_{p_1\times p_1}\left(\begin{bmatrix}
					Q^L_2&Q_{2,b}\\
			\end{bmatrix}^{\T}\right)^{\T}
			\\
			\hspace{-1cm}		+\toepU_{p_1\times p_2}\left(\begin{bmatrix}
					Q_3\\0_{p_1\times p_2}\\
					\end{bmatrix}\right)(I_{t}\otimes \left[W+WY^LY W\right])^{-1}\toepU_{p_1\times p_2}\left(\begin{bmatrix}
					Q^L_3&0_{p_2\times p_1}\\
			\end{bmatrix}^{\T}\right)^{\T}
	,
			\end{multlined}
\end{equation}
where the following equations are solvable and $Q_{2,b},W+WY^LYW$ are nonsingular:
\begin{align}
	Q^L_3\left(I_{p_1(t-1)}-D_{t-1}D^L_{t-1}-T_{t-1}T^L_{t-1}\right)&=D^L_{t-1},
	\\
	\left(I_{p_1t}-D_{t-1}D^L_{t-1}-T_{t-1}T^L_{t-1}\right)Q_3&=D_{t-1},
\qquad W=-(I_{p_2}+Q^L_3D_{t-1}),
	\\
	\begin{bmatrix}
		Q^L_2& Q_{2,b}
	\end{bmatrix}(I_{p_1t}-T_tT^L_t)
	&=
		\begin{bmatrix}
			0 & I_{p_1}\\
	\end{bmatrix},\\
	(I_{p_1t}-T_tT^L_t)
	\begin{bmatrix}
		Q_2\\ Q_{2,b}
		\end{bmatrix}&=
		\begin{bmatrix}
			0 \\ I_{p_1}\\
	\end{bmatrix},
	\qquad Q_{2,b}\in \mathbb{R}^{p_1\times p_1}.
\end{align}
\end{subequations}
\end{lemma}
\begin{proof}
	First consider the case $p_1=p_2=p$.
	Since
	\[
	I_{pt}-T_tT^L_t=\begin{bmatrix}
		I_{p}-Y Y^L & -Y D^L_{t-1} \\ -D_{t-1}Y^L & I_{p(t-1)}-D_{t-1}D^L_{t-1}-T_{t-1}T^L_{t-1}
\end{bmatrix},
\]
and
\begin{align*}
	(I_{pt}-T_tT^L_t)-Z_{t,p}(I_{pt}-T_tT^L_t)Z_{t,p}^{\T}
	&=I_{pt}-Z_{t,p}Z_{t,p}^{\T}-T_tT^L_t+Z_{t,p}T_tT^L_tZ_{t,p}^{\T}
	\\&=\begin{bmatrix}
		I_{p} & \\ & 0
	\end{bmatrix}-\begin{bmatrix}
		Y  & 0\\ D_{t-1} & T_{t-1}
	\end{bmatrix}\begin{bmatrix}
		Y^L & D^L_{t-1}\\ 0 & T^L_{t-1}
	\end{bmatrix}+\begin{bmatrix}
	0 & 0\\  T_{t-1} & 0\\
	\end{bmatrix}\begin{bmatrix}
		0& T^L_{t-1} \\0&0\\  
	\end{bmatrix}
	\\&
	=\begin{bmatrix}
	I_{p}-Y Y^L & -Y D^L_{t-1} \\ -D_{t-1}Y^L & -D_{t-1}D^L_{t-1}
	\end{bmatrix}
	,
\end{align*}
we have $\alpha_+(I_{pt}-T_tT^L_t,p)=2$.
By \cref{lm:form-of-inverse}, since $\alpha=1$, the case falls in Item~2 with substitutions
\[
	D_3\leftarrow 0, S_3\leftarrow I_{p},D_1\leftarrow D_{t-1},\Sigma\leftarrow -I_{p(t-1)},D_2^{\T}\leftarrow D^L_{t-1}, S_1^{\T}\leftarrow Y, S_2\leftarrow Y^L.
\]
Then \cref{eq:lm:form-of-inverse:1} (or equivalently \cref{eq:lm:form-of-inverse:1'}) becomes
\[
	\begin{aligned}[b]
(I_{pt}-T_tT^L_t)^{-1}
				&=
				\begin{multlined}[t]
				\toepU_{p\times p}\left(\begin{bmatrix}
						Q_2\\Q_{2,b}\\
						\end{bmatrix}\right)(I_{t}\otimes Q_{2,b})^{-1}\toepU_{p\times p}\left(\begin{bmatrix}
						Q^L_2&Q_{2,b}\\
				\end{bmatrix}^{\T}\right)^{\T}
				\\+\toepU_{p\times p}\left(\begin{bmatrix}
						Q_3\\0\\
						\end{bmatrix}\right)(I_{t}\otimes [W+WY^LYW])^{-1}\toepU_{p\times p}\left(\begin{bmatrix}
						Q^L_3&0\\
				\end{bmatrix}^{\T}\right)^{\T}
		,\end{multlined}
		\end{aligned}
	\]
	where $Q_2,Q_{2,b},Q^L_2,Q_3,Q^L_3,W$ is as in \cref{eq:lm:form-of-inverse:practical}.

Then consider the case $p_1>p_2$. Complement $Y$ to a $p_1\times p_1$ matrix $\wtd Y=\begin{bmatrix}
	Y & 0
\end{bmatrix}$ and similarly for $\wtd D_{t-1}=\begin{bmatrix}
D_{t-1} & 0
\end{bmatrix}, \wtd Y^L=\begin{bmatrix}
	Y^L\\0
\end{bmatrix}, \wtd D^L_{t-1}=\begin{bmatrix}
D^L_{t-1} \\ 0
\end{bmatrix}$. Immediately we are able to use the result above on the case $p_1=p_2$ to obtain
$
	\left[I_{p_1t}-\toepL_{p_1\times p_1}\left(\begin{bmatrix}\wtd Y\\\wtd D_{t-1}\end{bmatrix}\right)\toepL_{p_1\times p_1}\left(\begin{bmatrix}\wtd Y^L&\wtd D^L_{t-1}\end{bmatrix}^{\T}\right)^{\T}\right]^{-1}$.
Note that 
\begin{align*}
	\toepL_{p_1\times p_1}\left(\begin{bmatrix}\wtd Y\\\wtd D_{t-1}\end{bmatrix}\right)\toepL_{p_1\times p_1}\left(\begin{bmatrix}\wtd Y^L&\wtd D^L_{t-1}\end{bmatrix}^{\T}\right)^{\T}
	&=\begin{bmatrix}
			* & 0 & * & 0 & \cdots\\
			\vdots & \vdots & \vdots & \vdots & \\
			* & 0 & * & 0 & \cdots\\
		\end{bmatrix}\begin{bmatrix}
		* &\cdots & *\\
		0 &\cdots & 0\\
		* &\cdots & *\\
		0 &\cdots & 0\\
			\vdots &  &\vdots  \\
		\end{bmatrix}
		\\&=\begin{bmatrix}
			* & * &  \cdots\\
			\vdots &  \vdots & \\
			* &  * &  \cdots\\
		\end{bmatrix}\begin{bmatrix}
		* &\cdots & *\\
		* &\cdots & *\\
			\vdots &  &\vdots  \\
		\end{bmatrix}
		\\&=
	\toepL_{p_1\times p_2}\left(\begin{bmatrix} Y\\ D_{t-1}\end{bmatrix}\right)\toepL_{p_1\times p_2}\left(\begin{bmatrix} Y^L& D^L_{t-1}\end{bmatrix}^{\T}\right)^{\T}
		.
\end{align*}
Thus, $\wtd Q_2=Q_2,\wtd Q^L_2=Q^L_2,\wtd Q_{2,b}=Q_{2,b}$, and
$\wtd Q_3=\begin{bmatrix}
	Q_3 & 0
\end{bmatrix}, \wtd Q^L_3=\begin{bmatrix}
	Q^L_3\\0
\end{bmatrix}$.
Therefore,
\begin{align*}
	\wtd W&= -I_{p_1} - \begin{bmatrix}
		Q^L_3 \\ 0
	\end{bmatrix}\begin{bmatrix}
	D_{t-1} & 0
	\end{bmatrix}=\begin{bmatrix}
	-I_{p_2}-Q^L_3D_{t-1} & 0\\
	0 & -I_{p_1-p_2}\\
	\end{bmatrix}=\begin{bmatrix}
	W & \\ & -I_{p_1-p_2}
	\end{bmatrix},
	\\
	\wtd W\wtd Y^L\wtd Y\wtd W&=\begin{bmatrix}
	W & \\ & -I_{p_1-p_2}
\end{bmatrix}\begin{bmatrix}
	Y^L\\0
\end{bmatrix}\begin{bmatrix}
Y& 0
\end{bmatrix}\begin{bmatrix}
	W & \\ & -I_{p_1-p_2}
\end{bmatrix}
=\begin{bmatrix}
	WY^LYW & \\ & 0
\end{bmatrix}.
\end{align*}
Hence
\begin{align*}
	\MoveEqLeft
	\toepU_{p_1\times p_1}\left(\begin{bmatrix}
				\wtd Q_3\\0\\
				\end{bmatrix}\right)(I_{t}\otimes [\wtd W+\wtd W\wtd Y^L\wtd Y\wtd W])^{-1}\toepU_{p_1\times p_1}\left(\begin{bmatrix}
				\wtd Q^L_3&0\\
		\end{bmatrix}^{\T}\right)^{\T}
		\\&=\begin{bmatrix}
			* & 0 & * & 0 & \cdots\\
			\vdots & \vdots & \vdots & \vdots & \\
			* & 0 & * & 0 & \cdots\\
		\end{bmatrix}\begin{bmatrix}
		* & &&&\\
		& -I &&&\\
		&&* &&\\
		 &&& -I &\\
		 &&&&\ddots\\
		\end{bmatrix}\begin{bmatrix}
		* &\cdots & *\\
		0 &\cdots & 0\\
		* &\cdots & *\\
		0 &\cdots & 0\\
			\vdots &  &\vdots  \\
		\end{bmatrix}
		\\&=\begin{bmatrix}
			* & * &  \cdots\\
			\vdots &  \vdots & \\
			* &  * &  \cdots\\
		\end{bmatrix}\begin{bmatrix}
		* &&\\
		& *&\\
		&& \ddots \\
		\end{bmatrix}\begin{bmatrix}
		* &\cdots & *\\
		* &\cdots & *\\
			\vdots &  &\vdots  \\
		\end{bmatrix}
		\\&=
	\toepU_{p_1\times p_2}\left(\begin{bmatrix}
				Q_3\\0\\
				\end{bmatrix}\right)(I_{t}\otimes [ W+ W Y^L Y W])^{-1}\toepU_{p_1\times p_2}\left(\begin{bmatrix}
				Q^L_3&0\\
		\end{bmatrix}^{\T}\right)^{\T}
		.
\end{align*}

Finally consider the case $p_1<p_2$. Complement $Y$ to a $p_2\times p_2$ matrix $\wtd Y=\begin{bmatrix}
	Y \\ 0
\end{bmatrix}$ and similarly for $\wtd D_{t-1}^{\T}=\begin{bmatrix}
* & 0 & * & 0&\cdots
\end{bmatrix}$ where $D_{t-1}^{\T}=\begin{bmatrix}
	 * & * & \cdots
\end{bmatrix}$, and $\wtd Y^L=\begin{bmatrix}
Y^L&0
\end{bmatrix}, \wtd D^L_{t-1}=\begin{bmatrix}
* & 0 & * & 0&\cdots
\end{bmatrix}$ where $D^L_{t-1}=\begin{bmatrix}
	 * & * & \cdots
\end{bmatrix}$.
To make things clear, two permutations $P,P_s$ are used to make 
$P\begin{bmatrix}
	\wtd Y\\
	\wtd D_{t-1}
\end{bmatrix}=\begin{bmatrix}
	Y\\ D_{t-1} \\ 0
\end{bmatrix}, P_s\wtd D_{t-1}=\begin{bmatrix}
	D_{t-1} \\ 0
\end{bmatrix}$.
So 
$\begin{bmatrix}
 	\wtd Y^L & \wtd D^L_{t-1}
 \end{bmatrix}P^{\T}=\begin{bmatrix}
 Y^L & D^L_{t-1} & 0
 \end{bmatrix}, 
 	\wtd D^L_{t-1}
 P_s^{\T}=\begin{bmatrix}
  D^L_{t-1} & 0
 \end{bmatrix}$, and
 \begin{align*}
 	P\toepL_{p_2\times p_2}\left(\begin{bmatrix}\wtd Y\\\wtd D_{t-1}\end{bmatrix}\right)
	&=\begin{bmatrix}
			\toepL_{p_1\times p_2}\left(\begin{bmatrix}
		Y\\ D_{t-1} 
	\end{bmatrix}\right)\\ 0\\
		\end{bmatrix}
		,
		\\
	\toepL_{p_2\times p_2}\left(\begin{bmatrix}\wtd Y^L&\wtd D^L_{t-1}\end{bmatrix}^{\T}\right)^{\T}P^{\T}
	&=\begin{bmatrix}
			\toepL_{p_1\times p_2}\left(\begin{bmatrix}
					Y^L & D^L_{t-1} 
			\end{bmatrix}^{\T}\right)^{\T}& 0\\
		\end{bmatrix}
		.
 \end{align*}
Then we use the result above on the case $p_1=p_2$ to obtain $
\left[I_{p_2t}-\toepL_{p_2\times p_2}\left(\begin{bmatrix}\wtd Y\\\wtd D_{t-1}\end{bmatrix}\right)\toepL_{p_2\times p_2}\left(\begin{bmatrix}\wtd Y^L&\wtd D^L_{t-1}\end{bmatrix}^{\T}\right)^{\T}\right]^{-1}$.
Note that
\[
	\begin{multlined}[t]
	P\left[I_{p_2t}-\toepL_{p_2\times p_2}\left(\begin{bmatrix}\wtd Y\\\wtd D_{t-1}\end{bmatrix}\right)\toepL_{p_2\times p_2}\left(\begin{bmatrix}\wtd Y^L&\wtd D^L_{t-1}\end{bmatrix}^{\T}\right)^{\T}\right]P^{\T}
	\\=\begin{bmatrix}
		I_{p_1t}-\toepL_{p_1\times p_2}\left(\begin{bmatrix}Y\\ D_{t-1}\end{bmatrix}\right)\toepL_{p_1\times p_2}\left(\begin{bmatrix}Y^L& D^L_{t-1}\end{bmatrix}^{\T}\right)^{\T} & \\ & I_{(p_2-p_1)t}
	\end{bmatrix}.
	\end{multlined}
\]
Thus,
\begin{align*}
	P\begin{bmatrix}
		\wtd Q_2\\
		\wtd Q_{2,b}\\
	\end{bmatrix}
	&=P\left[I_{p_2t}-\toepL_{p_2\times p_2}\left(\begin{bmatrix}\wtd Y\\\wtd D_{t-1}\end{bmatrix}\right)\toepL_{p_2\times p_2}\left(\begin{bmatrix}\wtd Y^L&\wtd D^L_{t-1}\end{bmatrix}^{\T}\right)^{\T}\right]^{-1}P^{\T}P\begin{bmatrix}
		0\\ I_{p_2}\\
	\end{bmatrix}
	\\&=\begin{bmatrix}
		\left[ I_{p_1t}-\toepL_{p_1\times p_2}\left(\begin{bmatrix}Y\\ D_{t-1}\end{bmatrix}\right)\toepL_{p_1\times p_2}\left(\begin{bmatrix}Y^L& D^L_{t-1}\end{bmatrix}^{\T}\right)^{\T} \right]^{-1} & \\ & I_{(p_2-p_1)t}
		\end{bmatrix}
		\begin{bmatrix}
			0\\\begin{bmatrix}
				I_{p_1} & 0
			\end{bmatrix}\\0\\\begin{bmatrix}
			0 & I_{p_2-p_1}
			\end{bmatrix}
		\end{bmatrix}
		\\&=\begin{bmatrix}
		\left[ I_{p_1t}-\toepL_{p_1\times p_2}\left(\begin{bmatrix}Y\\ D_{t-1}\end{bmatrix}\right)\toepL_{p_1\times p_2}\left(\begin{bmatrix}Y^L& D^L_{t-1}\end{bmatrix}^{\T}\right)^{\T} \right]^{-1}\begin{bmatrix}
			0 \\ I_{p_1}
		\end{bmatrix} & 0\\
		0 &\begin{bmatrix}
			0 \\ I_{p_2-p_1}
		\end{bmatrix}
	\end{bmatrix}
	\\&
	=\begin{bmatrix}
		\begin{bmatrix}
			Q_2\\ Q_{2,b}
		\end{bmatrix} & 0\\
		0 &\begin{bmatrix}
			0 \\ I_{p_2-p_1}
		\end{bmatrix}
	\end{bmatrix}
	,
\end{align*}
and similarly,
	$\begin{bmatrix}
		\wtd Q^L_2&
		\wtd Q_{2,b}\\
	\end{bmatrix}P^{\T}
		=\begin{bmatrix}
		\begin{bmatrix}
			Q^L_2 & Q_{2,b}
		\end{bmatrix} & 0\\
		0 &\begin{bmatrix}
			0 & I_{p_2-p_1}
		\end{bmatrix}
	\end{bmatrix}
	$.
	Therefore,
$
\wtd Q_{2,b}=\begin{bmatrix}
		Q_{2,b} & 0 \\
			0 & I_{p_2-p_1}
\end{bmatrix}$ and
\begin{align*}
	\MoveEqLeft[0] P\toepU_{p_2\times p_2}\left(\begin{bmatrix}
			\wtd Q_2\\ \wtd Q_{2,b}\\
			\end{bmatrix}\right)P^{\T}P(I_{t}\otimes \wtd  Q_{2,b})^{-1}P^{\T}P\toepU_{p_2\times p_2}\left(\begin{bmatrix}
			\wtd Q^L_2& \wtd Q_{2,b}\\
	\end{bmatrix}^{\T}\right)^{\T}P^{\T}
	\\&=
	\begin{bmatrix}
		\toepU_{p_1\times p_1}\left(\begin{bmatrix}
			Q_2\\ Q_{2,b}\\
	\end{bmatrix}\right) & 0\\ 0& I_{(p_2-p_1)t}
	\end{bmatrix}\begin{bmatrix}
	(I_{t}\otimes  Q_{2,b})^{-1} & \\ & I_{(p_2-p_1)t}
	\end{bmatrix}\begin{bmatrix}
		\toepU_{p_1\times p_1}\left(\begin{bmatrix}
				Q^L_2& Q_{2,b}\\
		\end{bmatrix}^{\T}\right) & 0\\ 0& I_{(p_2-p_1)t}
	\end{bmatrix}^{\T}
	\\&=
	\begin{bmatrix}
		\toepU_{p_1\times p_1}\left(\begin{bmatrix}
			Q_2\\ Q_{2,b}\\
	\end{bmatrix}\right)(I_{t}\otimes  Q_{2,b})^{-1}\toepU_{p_1\times p_1}\left(\begin{bmatrix}
				Q^L_2& Q_{2,b}\\
		\end{bmatrix}^{\T}\right)^{\T} & \\ & I_{(p_2-p_1)t}
	\end{bmatrix}
	.
\end{align*}
Similarly,
\begin{align*}
	P_s\wtd Q_3
	&=P_s\left[I_{p_2(t-1)}-\wtd D_{t-1}\wtd D^L_{t-1}-\wtd T_{t-1}\wtd T^L_{t-1}\right]^{-1}P_s^{\T}P_s\wtd D_{t-1}
	\\&
	=\begin{bmatrix}
		\left[I_{p_1(t-1)}- D_{t-1} D^L_{t-1}- T_{t-1} T^L_{t-1}\right]^{-1} & \\ & I_{(p_2-p_1)(t-1)}
		\end{bmatrix}\begin{bmatrix}
		D_{t-1}\\0
		\end{bmatrix}
		=\begin{bmatrix}
			Q_3 \\ 0 
	\end{bmatrix}
	,
\end{align*}
and similarly,
$
 	\wtd Q^L_3P_s^{\T}
 		=\begin{bmatrix}
 			Q^L_3 & 0 
 	\end{bmatrix}
 	,
 $
 and then
\begin{align*}
	\wtd W &= -I_{p_2} - \wtd Q_3^LP_s^{\T}P_s\wtd D_{t-1}=-I_{p_2}-Q_3^L D_{t-1}=W,
	\\
	\wtd W\wtd Y^L\wtd Y\wtd W&=W\begin{bmatrix}
		Y^L&0
\end{bmatrix}\begin{bmatrix}
Y\\ 0
\end{bmatrix}W
= WY^LYW.
\end{align*}
Hence
\begin{align*}
	\MoveEqLeft P\toepU_{p_2\times p_2}\left(\begin{bmatrix}
			\wtd Q_3\\ 0\\
			\end{bmatrix}\right)(I_{t}\otimes [\wtd  W+\wtd W\wtd Y^L\wtd Y\wtd W])^{-1}\toepU_{p_2\times p_2}\left(\begin{bmatrix}
			\wtd Q^L_3& 0\\
	\end{bmatrix}^{\T}\right)^{\T}P^{\T}
	\\&=
	\begin{bmatrix}
		\toepU_{p_1\times p_2}\left(\begin{bmatrix}
			Q_3\\ 0\\
	\end{bmatrix}\right) \\ 0
	\end{bmatrix}
	(I_{t}\otimes  [  W+ W Y^L Y W])^{-1}
	\begin{bmatrix}
		\toepU_{p_1\times p_2}\left(\begin{bmatrix}
			Q^L_3& 0\\
	\end{bmatrix}^{\T}\right) \\ 0
	\end{bmatrix}^{\T}
	\\&=
	\begin{bmatrix}
		\toepU_{p_1\times p_2}\left(\begin{bmatrix}
			Q_3\\ 0\\
	\end{bmatrix}\right)(I_{t}\otimes  [  W+ W Y^L Y W])^{-1}\toepU_{p_1\times p_2}\left(\begin{bmatrix}
			Q^L_3& 0\\
	\end{bmatrix}^{\T}\right)^{\T} & \\ & 0
	\end{bmatrix}
	.
	\qedhere
\end{align*}
\end{proof}
Finally, \cref{lm:form-of-inverse:dare} comes out as a corollary.
\begin{proof}[Proof of \cref{lm:form-of-inverse:dare}]
	Use \cref{lm:form-of-inverse:practical} with $Y=-(Y^L)^{\T}, D_{t-1}=-(D^L_{t-1})^{\T}$.
	Then we take $Q_{2,b} =Q_1$ to obtain the result.
\end{proof}

{\small
	\bibliographystyle{plain}
	\bibliography{../strings,sdals}
}

\end{document}